\documentclass{amsart}
\usepackage{graphics}
\usepackage{amsmath,amsthm}
\usepackage{amsfonts,amssymb}
\usepackage{color}
\usepackage{comment}
\usepackage{mathtools}
\usepackage{enumitem}
\usepackage{soul}
\usepackage[normalem]{ulem}
\usepackage{thmtools,thm-restate}

\newtheorem{thm}{Theorem}[section]

\newtheorem{lem}[thm]{Lemma}

\theoremstyle{remark}
\newtheorem{rem}{Remark}[section]

\numberwithin{equation}{section}

\newcommand{\be}{\begin{equation}}
\newcommand{\ee}{\end{equation}}
\newcommand{\bea}{\begin{eqnarray}}

\newcommand{\eea}{\end{eqnarray}}
\newcommand{\Bea}{\begin{eqnarray*}}
\newcommand{\Eea}{\end{eqnarray*}}

\def\CL{{\mathcal L}}

\def\C{{\mathbb C}}
\def\H{{\mathbb H}}
\def\N{{\mathbb N}}
\def\R{{\mathbb R}}

\def\Z{{\mathbb Z}}
\def\T{{\mathbb T}}

\def\1{\text{\bf {1}}}

\begin{document}

\title[Localisation of Spectral Sums]
{Localisation of Spectral Sums corresponding to the sub-Laplacian on the Heisenberg Group}

\author{Rahul Garg}
\address{Department of Mathematics, Indian Institute of Science Education and Research Bhopal, India.} 
\email{rahulgarg@iiserb.ac.in}

\author{K. Jotsaroop}
\address{Department of Mathematics, Indian Institute of Science Education and Research Mohali, India.} 
\email{jotsaroop@iisermohali.ac.in}

\subjclass[2010]{Primary 43A50, 43A80; Secondary 26D10, 42B10, 46E35}


\dedicatory{Dedicated to Prof. Sundaram Thangavelu on the occasion of his 60th birthday.}

\keywords{Heisenberg Group, sub-Laplacian, Bochner Riesz means, Hardy-Sobolev inequality, interpolation}

\begin{abstract}
In this article we study localisation of spectral sums $\{S_R\}_{R > 0}$ associated to the sub-Laplacian $\CL$ on the Heisenberg Group $\H^d$ where $S_Rf := \int_0^R dE_{\lambda }f$, with $\CL = \int_0^{\infty} \lambda \, dE_{\lambda}$ being the spectral resolution of $\CL.$ We prove that for any compactly supported function $f \in L^2(\H^d)$, and for any $\gamma < \frac{1}{2}$, $R^{\gamma} S_R f  \to 0$ as $ R \to \infty$, almost everywhere off $supp (f)$. 
\end{abstract}

\maketitle

\section{Introduction}
We define the standard Laplacian on $\R^d$ as $\Delta = -\sum_{j=1}^d \frac{\partial^2}{\partial x_j^2}$. It is a densely defined positive, self-adjoint operator and it admits a spectral decomposition in $L^2(\R^d).$ Using functional calculus we define  Bochner Riesz means of order $\alpha\geq 0$ corresponding to  $\Delta$ by 
$$S_R^{\alpha} = \left(I-\frac{1}{R^2}\Delta\right)_{+}^{\alpha}.$$ 

When $\alpha=0,$ we obtain the spectral sums $\{S_R^0\}_{R > 0}$ associated to $\Delta.$ Let $f$ be a measurable function on $\R^d$ vanishing identically on an open subset $\Omega$ of $\R^d.$ We say that the localisation principle for $S_R^{\alpha}$ holds if $\lim_{R \rightarrow \infty} S_R^{\alpha}f=0$ either pointwise or uniformly over compact subsets of $\Omega$. To be more specific, when the convergence is pointwise a.e. we will refer to it as a.e. localisation principle. 

For any $d \geq 1$ and $\alpha  \geq \frac{d-1}{2},$  it is well known that $S_R^{\alpha}f$ converges to $0$ uniformly on every compact subset of $\Omega$ for any $f \in L^p(\R^d)$ with $1\leq p < \infty$ (see \cite{SB} and Chapter 7 of \cite{SW}). In higher dimensions $d \geq 2$, in the case when $\alpha=0$, A. I. Bastis \cite{AIB} and P. Sjolin \cite{PS} proved a.e. localisation principle for \emph{compactly supported} functions. Later, Bastis \cite{AIB2} established the a.e. localisation principle for all functions $ f\in L^2(\R^d)$. Around the same time, A. Carbery and F. Soria \cite{CS} proved that the a.e. localisation principle is indeed true for functions $f\in L^p(\R^d)$ with $2\leq p<\frac{2d}{d-1}$. Using dilation and translation, and thus assuming without loss of generality that $f$ is identically $0$ in the open ball $\{|x| < 3\}$, Carbery-Soria  \cite{CS} proved their result on the a.e. localisation principle by showing that for any $\beta <1$, the following weighted-norm inequality holds:
\begin{equation}
\label{0.1}\int_{|x|\leq 1}\sup_{R>1}|S_R^0 f(x)|^2 \, dx \leq C_\beta \int_{|x|\geq 3} \frac{|f(x)|^2}{|x|^{\beta}} \, dx. 
\end{equation} 

One could observe that the above inequality is stronger in the sense that it gives the a.e. localisation principle for a larger class of functions, as for any $2 \leq p < \frac{2d}{d-1}$, one has $L^p(\R^d\setminus \{|x| < 3\}, dx) \xhookrightarrow{} L^2(\R^d\setminus \{|x| < 3\}, |x|^{-\beta} \, dx)$ for any $\beta > d(1-2/p)$, by H\"older's inequality. 

Recently, R. Ashurov \cite{RA} proved the a.e. localisation principle for the spectral sums corresponding to the Laplacian on $d$-dimensional Torus $\T^d$ for $f \in L^2(\T^d).$ 

Let $\H^d$ be the  Heisenberg group. We consider the positive sub-Laplacian (or Kohn-Laplacian) $\CL$ on $\H^d$. It is known that $\CL$ has a self-adjoint extension on $L^2(\H^d)$. Therefore, it admits a spectral resolution of identity such that $\CL=\int_0^{\infty}\lambda dE_{\lambda}$. We define the spectral sums of $f$ corresponding to $\CL$ as $S_Rf=\int_0^R dE_{\lambda}f.$ In this article, we study a.e. localisation of these spectral sums $S_Rf.$ 

Before stating our main result, let us also introduce the Bochner Riesz means of order $\alpha\geq 0,$ corresponding to the sub-Laplacian $\mathcal{L}$, which include the spectral sums $\{S_R\}_{R > 0}$ as a special case. As we did in the case of $\Delta$ on $\R^d,$ using the functional calculus of $\CL$ we define the Bochner Riesz means of order $\alpha\geq 0$ for $\CL$ as $$S_R^{\alpha} = \left(I-\frac{1}{R^2}\CL\right)_{+}^{\alpha}.$$

Note that $\alpha=0$ corresponds to the operator $S_R$ (that is, $S_R = S_R^0$). When $\alpha>0$, D. Gorges and D. M\"uller \cite{GM} proved that $S_R^{\alpha}f$ converges a.e. to $f$ as $R \to \infty$ for all $f\in L^p(\H^d), \frac{Q-1}{Q} \left(\frac{1}{2} - \frac{\alpha}{D-1}\right) < \frac{1}{p}\leq \frac{1}{2},$ where $Q= 2d+2$ and $D= 2d+1.$
 
In this paper we are interested in establishing a.e. localisation principle for the spectral sums $\{S_R\}_{R > 0}$ associated to $\CL.$ We show that the following version of localisation principle holds for compactly supported functions in $L^2(\H^d)$. 

\begin{thm}\label{main}
Let $f \in L^2(\H^d)$ be compactly supported. Then, for any $\gamma<\frac{1}{2}$, $R^{\gamma} S_R f \to 0$ as $ R \to \infty$, almost everywhere off $supp (f)$.  
\end{thm}

We remark here that our method does not extend to functions in $L^2(\mathbb{H}^d)$ that are not necessarily compactly supported. We expect though that the above result should hold without the assumption of compact support for functions. 

Let us also mention some of the important developments on a.e. convergence of spectral sums associated to the sub-Laplacian on other groups. In \cite{MMP}, the authors proved a.e. convergence of spectral sums for the right invariant sub-Laplacian $\CL$ on a connected Lie Group. In fact, there they proved an analogue of Rademacher-Menshov theorem \cite{M,R} for general Lie groups. More precisely, they showed that $S_Rf\rightarrow f$ a.e. as $R\rightarrow\infty$  for any $f$ such that $\log(2+\CL)f\in L^2(G)$. See also \cite{LC} for a similar result for spectral sums for the  standard Laplacian $\Delta$ on $\R^d$. 

\medskip \noindent \textbf{Organisation of the paper:} We recall some preliminaries of the Heisenberg group in Section \ref{prelim}. In Section \ref{proofs}, we state and prove some sharp weighted estimates for the  spectral sums (Theorem \ref{A1}) from which the a.e. localisation principle (Theorem \ref{main}) follows immediately. We also state the key lemma (Lemma \ref{B}) in Section \ref{proofs} that is required in establishing Theorem \ref{A1}, and the same is proved in Section \ref{key-estimates}. In the proof of Lemma \ref{B}, we make use of an interpolation of certain Besov-type spaces and we prove it separately in the Appendix (Section \ref{appendix}). 

\medskip \noindent \textbf{Notation:} For $A, B >0$, the expression $A \lesssim B$ indicates that $A \leq C B$ for some $C > 0$. We write $A \lesssim_{\beta} B$ when the implicit constant $C$ may depend on $\beta$. We also use the notation $A \sim B$ if $A \lesssim B$ and $B \lesssim A$. 
The constants appearing in the inequalities may change from line to line. We only keep track of the dependence of the constants.

\section{Preliminaries}\label{prelim}
We recall basics of the Heisenberg group. 
The Heisenberg Group $\H^d$ can be identified with $\C^{d}\times\R.$ The Haar measure on $\H^d$ is given by the Lebesgue measure on $\C^d\times\R.$ 
We denote a point in $\H^d$ by $(z,t)$, where $z = (z_1, z_2, \ldots, z_d)\in\C^d;$ $t\in\R,$ and $z_j= x_j + i y_j;$ $x_j, y_j \in \R.$
It is well known that $\H^d$ is a simply connected, unimodular, two step nilpotent Lie group under the group operation 
$$(z,t)(w,s) = (z+w, t+s +\frac{1}{2} Im(z \cdot \bar{w})),$$ 
where $z \cdot \bar{w}= \sum_{j=1}^{d} z_j \bar{w_j}.$  
The convolution of functions on $\H^d$ is given by 
$$f \ast g(z,t) = \int_{\H^d} f((z,t)(w,s)^{-1}) g(w,s) \, dw \, ds.$$ 

We consider the following left-invariant vector fields on $\H^d$: 
\begin{eqnarray*} 
X_j = \frac{\partial}{\partial x_j} + \frac{1}{2} y_j \frac{\partial}{\partial t}, \quad Y_j = \frac{\partial}{\partial y_j} - \frac{1}{2} x_j \frac{\partial}{\partial t}, \quad T = \frac{\partial}{\partial t}, 
\end{eqnarray*} 
for $1 \leq j \leq d.$ These form a basis of the vector space of left invariant vector fields on $\H^d.$ In fact, the Lie algebra of left invariant vector fields on $\H^d$ is generated by taking the Lie brackets and finite linear combinations of $\{ X_j, Y_j \}_{j=1}^{d}.$ 
The operator
$$\CL := -\sum_{j=1}^d (X_j^2 + Y_j^2),$$ 
is known as the sub-Laplacian (or Kohn-Laplacian) on $\H^d.$  
It is known that $\CL$ is a densely defined positive, hypoelliptic operator and it has a self-adjoint extension on $L^2(\H^d),$ and that it also commutes with left translations.
The spectrum of $\CL$ is well known (see, for example, Section 2.1 of \cite{ST}). 

The action of $\mathcal{L}$ on functions of the form $f(z,t) = e^{i\lambda t} g(z)$, $\lambda \neq 0$, leads us to the following family of operators defined by $\mathcal{L}f(z,t) = e^{i\lambda t} L(\lambda)g(z).$ More explicitly,
$$L(\lambda) = \Delta_{\R^{2d}} + \frac{1}{4} \lambda^2 |z|^2 + i \lambda \sum_{j=1}^d \left(x_j \frac{\partial}{\partial y_j} - y_j \frac{\partial}{\partial x_j} \right).$$

These operators are called special Hermite operators and their spectral projections in $L^2(\C^d)$ are explicitly known (see, for example, Section 1.3 of \cite{ST1} and Section 1.4 of \cite{ST}). When $\lambda=0,$ $L(\lambda)$ reduces to $\Delta$ on $\R^{2d}$. 

Fix $\lambda \neq 0$. Let $\varphi_k$ denote the Laguerre functions of order $d-1.$ Then, we know that 
$$ L(\lambda) \varphi_k(\sqrt{|\lambda|}  \, \cdot) = |\lambda|(2k+d) \varphi_{k}(\sqrt{|\lambda|} \, \cdot).$$

If we write $E_{k,\lambda}(z,t)= e^{i\lambda t} \varphi_k(\sqrt{|\lambda|}z)$, then using the left invariance of $\CL,$ it follows that for $f\in \mathcal{S}(\H^d),$ the space of Schwartz class functions on $\mathbb{H}^d$, we have 
$$\CL(f* E_{k,\lambda}) = f* \CL E_{k,\lambda} = |\lambda|(2k+d)f* E_{k,\lambda}.$$ 

In terms of the spectral resolution of identity associated to $\CL$, we have that $f\in L^2(\H^d)$ can be written as 
$$f(z,t) = \frac{1}{(2\pi)^{d+1}} \sum_{k\geq 0} \int_{\R\setminus{\{0\}}} f*E_{k,\lambda}(z,t) |\lambda|^d \, d\lambda$$ 
in $L^2$-sense.
Note that $f\rightarrow f*E_{k,\lambda}+ f*E_{k,-\lambda}$ is the spectral projection corresponding to the eigenvalue $|\lambda|(2k+d).$ 

Therefore, it suffices to work on the spectrum 
$$\{|\lambda|(2k+d): \lambda \in \R\setminus \{0\}, k\in \N \cup \{0\}\}$$
as $\lambda=0$ corresponds to the set of measure zero in the spectral resolution of $\CL$ given above.

Recall that the Bochner Riesz means of order $\alpha\geq 0,$ corresponding to $\CL$ are given by
\begin{equation} \label{BR-subLap-def}
S_R^{\alpha}f := \left(I-\frac{1}{R^2}\CL \right)_{+}^{\alpha}f. 
\end{equation}

In terms of the spectral projections, the above can be expressed as 
\begin{equation} \label{BR-subLap-def-expansion} 
S_R^{\alpha}f(z,t) = \frac{1}{(2\pi)^{d+1}} \int_{\R\setminus\{0\}} \sum_{k\geq 0} \left(1 - \frac{(2k+d)|\lambda|}{R^2}\right)_{+}^{\alpha} f* E_{k,\lambda} (z,t) |\lambda|^d \, d\lambda, 
\end{equation}

For any $\lambda\in \R$, we denote by $f^\lambda$ the inverse Euclidean Fourier transform (upto a constant) of $f$ in the last variable at the point $\lambda,$ that is, $f^{\lambda}(z) = \int_{\R} f(z,t) e^{i \lambda t} \, dt$. It is easy to verify that $(f \ast g)^{\lambda} = f^{\lambda} \times_{\lambda} g^{\lambda}$, where $\times_{\lambda}$ is called the $\lambda-$twisted convolution on $\mathbb{C}^d$ and is defined by 
$$F \times_{\lambda} G (z) = \int_{\mathbb{C}^d} F(z-w)G(w) e^{i \frac{\lambda}{2}  Im(z \cdot \bar{w})} \, dw,$$
for any $F, G \in L^1(\mathbb{C}^d).$ 

For each $k \in \N$ and $\lambda\in \R\setminus \{0\}$, let us consider the operator $P_{k,\lambda}$ defined on $L^1 \cap L^2(\C^d)$ by 
\begin{align*} 
P_{k,\lambda}(F) (z) = F \times_{-\lambda} \varphi_k(\sqrt{|\lambda|} \cdot) (z) = \int_{\C^d} F(z-w)  \varphi_k(\sqrt{|\lambda|}(w)) e^{-i \frac{\lambda}{2} Im(z \cdot \bar{w})} \, dw. 
\end{align*} 

It is known (see Section $2.1$, page 53 in \cite{ST}) that each of $(2\pi)^{-d} |\lambda|^d P_{k,\lambda}$ extends to $L^2(\C^d)$ as an orthonormal projection, and for $F \in L^2(\C^d)$, we have $F = (2\pi)^{-d} \sum_{k \geq 0} |\lambda|^d P_{k,\lambda} (F)$ in $L^2$-sense. In fact, $(2\pi)^{-d} |\lambda|^d P_{k,\lambda}$ is the spectral projection of $L(\lambda)$ corresponding to the eigenvalue $(2k+d)|\lambda|$. 

Writing $e_k(z,t) = \int_{\mathbb{R}} |\lambda|^d E_{k,\lambda}(z,t) \, d \lambda = \int_{\mathbb{R}} |\lambda|^d \varphi_k(\sqrt{|\lambda|}z) e^{i \lambda t} \, d \lambda,$ one can verify that 
\begin{equation} \label{e_k}
f* E_{k,\lambda} (z,t) = (f \ast e_k)^{-\lambda} (z) e^{i \lambda t} = P_{k,\lambda}(f^{-\lambda}) (z) e^{i \lambda t} |\lambda|^{d}.
\end{equation} 

Therefore, one could also express $S_R^{\alpha}$ as follows: 
\begin{equation*} 
S_R^{\alpha}f(z,t) =  \frac{1}{(2\pi)^{d+1}} \int_{\R\setminus\{0\}}\sum_{k\geq 0}\left(1-\frac{(2k+d)|\lambda|}{R^2}\right)_{+}^{\alpha} P_{k,\lambda}(f^{-\lambda})(z) e^{i\lambda t}|\lambda|^d \, d\lambda. 
\end{equation*} 

By abuse of notation, we write $S_R^{0}$ as $S_R$. It is then clear from the above that 
\begin{equation} \label{spectralsums} 
S_R f(z,t) = \frac{1}{(2\pi)^{d+1}} \int_{\R\setminus\{0\}}\sum_{k\geq 0}\chi_{\{\sqrt{(2k+d)|\lambda|}\leq R\}} ((2k+d)|\lambda|)P_{k,\lambda}(f^{-\lambda})(z) e^{i\lambda t}|\lambda|^d \, d\lambda. 
\end{equation}

Before proceeding further, let us also recall certain functional identities which will be required in the proofs later. 

Define  $\phi_{\alpha,\beta}(z)=(\pi(z,0)\Phi_{\alpha},\Phi_{\beta}),$ the special Hermite functions on $\C^d$. Here $\{\Phi_{\alpha}\}_{\alpha\in \N^d}$ are the Hermite functions defined on $\R^d$ and $\pi$ is the Schr\"{o}dinger representation of the Heisenberg group. It is well known  that the system $\{\Phi_{\alpha}\}_{\alpha\in \N^d}$ forms an orthonormal basis of $L^2(\R^d)$. Using the properties of the Schr\"{o}dinger representation one can show that $|\lambda|^{d/2} \phi_{\alpha,\beta}(\sqrt{|\lambda|}\cdot)$ form an orthonormal basis of $L^2(\C^d)$ (see, for example, Theorem $1.3.2$ on page 16 in \cite{ST1}). We note the following identities: 
\begin{align} \label{laguerre-expansion} 
\varphi_k(\sqrt{|\lambda|}z) = (2\pi)^{d/2} \sum_{|\alpha|=k} \phi_{\alpha,\alpha}(\sqrt{|\lambda|}z), 
\end{align}
and 
\begin{align} \label{twisted}
\phi_{\mu,\nu} (\sqrt{|\lambda|} \, \cdot) \times_{\lambda} \phi_{\alpha,\beta} (\sqrt{|\lambda|} \, \cdot) = (2\pi)^{d/2} |\lambda|^{-d} \phi_{\mu,\beta} (\sqrt{|\lambda|} \, \cdot) \delta_{\nu,\alpha}. 
\end{align}

The above identities can be found in \cite{ST1} (Proposition $1.3.2$ on page 21 and identity $(1.3.42)$ on page 22). In Theorem $1.3.6$ of \cite{ST1}, the above mentioned identity \eqref{twisted} is proved for $\lambda=1$ only. For $\lambda \neq 0$, one can make a change of variables to reduce it to the case of $\lambda=1.$ 

\section{Proof of theorem \ref{main}}\label{proofs}
Let $\|\cdot \|$ be the homogeneous Cygan-Kor\'anyi norm on $\H^d$ which is given by $\| (z,t) \| = \left(|z|^4 + 16 |t|^2\right)^{1/4} = \left(\left(\sum_{j=1}^d |z_j|^2\right)^2 + 16 |t|^2\right)^{1/4}$. It is known that $\|\cdot \|$ is subadditive on $\H^d$ (see, for example, \cite{Cyg}). We write the corresponding left invariant distance function $d_K((z,t), (w,s)): = \|(z,t)^{-1}(w,s)\|$. 
We also define the non-isotropic dilations $\{\delta_r\}_{r>0}$ on $\H^d$ as $\delta_r(z,t)=(rz,r^2t),$ where $rz = (rz_1, rz_2, \ldots, rz_d)$ for $z \in \C^d.$ It is easy to verify that $\|\delta_r(z,t)\|=r\|(z,t)\|$.

We denote by $d_{CC}$ the Carnot-Carath\'eodory metric on $\H^d$, which is also referred to as the control distance. It is well known (see, for example, Proposition 5.1.4 on page 230 and Theorem 5.2.8 on page 235 in \cite{BLU}) that $d_K$ and $d_{CC}$ are equivalent, that is, there exists a constant $A>1$ such that 
$$A^{-1} d_{CC}((z,t), (w,s)) \leq d_K((z,t), (w,s)) \leq A \, d_{CC}((z,t), (w,s)).$$ 

Note that for studying localisation principle, using non-isotropic dilation and left translation, one may assume without loss of generality that $f$ is identically $0$ in the open ball $\{(z,t) : \|(z,t)\| < 3\}$.

Let $0 \leq \psi \leq 1$ be an even and smooth function on $\R$ such that $\psi \equiv 1$ on the interval $[-\frac{\epsilon_0}{2}, \frac{\epsilon_0}{2}]$ and support of $\psi$ is contained in the interval $[-\epsilon_0, \epsilon_0]$, for some $0<\epsilon_0<1$, which will be chosen later. Denote the characteristic function of the interval $[-R, R]$ on $\R$ by $\chi_R$. Then, in $L^2$-sense, we have 
$$\chi_{ R}(\sqrt{\eta}) = {\frac{2}{\pi}} \int_{0}^{\infty} \frac{\sin(R\rho)}{\rho} \cos(\sqrt{\eta}\rho) \, d\rho.$$

Using the above expression for $\chi_R$ in \eqref{BR-subLap-def} (for $\alpha = 0$) 
one can see that for any $f\in\mathcal{S}(\H^d)$ we have
\begin{align}\label{SRf-L2sum} 
S_R f(z,t) &= {\frac{2}{\pi}} \int_{0}^{\infty} \frac{\sin(R\rho)}{\rho} \cos(\rho\sqrt{\CL})f(z,t) \, d\rho \\
\nonumber &= {\frac{2}{\pi}} \int_{0}^{\infty}\psi(\rho) \frac{\sin(R\rho)}{\rho} \cos(\rho\sqrt{\CL})f(z,t) \, d\rho \\
\nonumber & \quad+ {\frac{2}{\pi}} \int_{0}^{\infty}\left(1-\psi(\rho)\right)\frac{\sin(R\rho)}{\rho}\cos(\rho\sqrt{\CL})f(z,t) \, d\rho. 
\end{align} 

Note that $\cos(\rho \sqrt{\CL}) f$ is the solution to the wave equation for the sub-Laplacian $\CL$ with initial data $f$ and initial speed $0$. Since $f$ is identically $0$ on the ball $\{(z,t): d_{CC}((z,t), e) < 3 A^{-1}\}$, where $e$ is the identity element of $\H^d$, it follows from the finite speed of propagation of the wave equation for $\CL$ on $\H^d$ (see \cite{Melrose} and Theorem $6.2$ in the appendix of \cite{Muller}), that $\cos(\rho \sqrt{\CL}) f(z,t) \equiv 0$ on $\{(z,t) : d_{CC}((z,t), e) < 3 A^{-1} - \rho \}$ for any $0 \leq \rho \leq 3 A^{-1}$. As a consequence, choosing $\epsilon_0 = A^{-1}$, we get that $\cos(\rho \sqrt{\CL}) f \equiv 0$ on $\{(z,t) : \|(z,t)\| < 2 A^{-2}\}$ for any $0 \leq \rho \leq \epsilon_0$. 

Hence, for $f \equiv 0$ on the set $\{\|(z,t)\| < 3\}$, choosing $\epsilon_0 = A^{-1}$, it suffices to study 
\begin{align}\label{BRf-def} 
B_Rf(z,t)=  \int_{0}^{\infty}\left(1-\psi(\rho)\right) \frac{\sin(R\rho)}{\rho}\cos(\rho\sqrt{\CL})f(z,t) \, d\rho,
\end{align} 
for $(z,t)$ such that $\|(z,t)\| < 2A^{-2}$.

In order to prove our result, we further break $B_R$ into simpler operators and analyse each piece separately.
Let $\Psi$ be an even and smooth function on $\R$ supported in the interval $[-2,2]$, further with the property that $0 \leq \Psi \leq 1$ and that it is identically $1$ on the interval $[-1, 1]$. Restricting $\Psi$ on $[0, \infty),$ it is easy to verify that on $[0, \infty)$ one has
$$1 = \sum_{j=1}^\infty \widetilde{\psi}_j(\rho) + \Psi(\rho),$$
where $\widetilde{\psi}_j(\rho) = \Psi(2^{-j}\rho) - \Psi(2^{-j+1}\rho)$, for $j \geq 1$. Clearly, each $\widetilde{\psi}_j \in C_c^{\infty}([0,\infty))$ with support in $[2^{j-1}, 2^{j+1}]$. From this, we have 
$$1-\psi(\rho) = \sum_{j=1}^\infty\left(1-\psi(\rho)\right) \widetilde{\psi}_j(\rho) +\left(1-\psi(\rho)\right) \Psi(\rho).$$

Therefore, for any $f\in\mathcal{S}(\H^d),$ we get
\begin{align*} 
B_Rf(z,t) &=  \sum_{j\geq 1} \int_{0}^{\infty}\left(1-\psi(\rho)\right)\widetilde{\psi}_j(\rho)\frac{\sin(R\rho)}{\rho} \cos(\rho\sqrt{\CL})f(z,t) \, d\rho \\
&\quad \quad +  \int_{0}^{\infty}\left(1-\psi(\rho)\right)\Psi(\rho) \frac{\sin(R\rho)}{\rho}\cos(\rho\sqrt{\CL})f(z,t) \, d\rho. 
\end{align*} 

Now, since $1-\psi(\rho)\equiv 1$ on $[\epsilon_0, \infty),$ and is identically $0$ on $(0,\epsilon_0/2]$, it suffices to analyse $\sum_{j\geq 1} B_{R,j} f(z,t)$, where 
$$ B_{R,j} f(z,t) = \int_{0}^{\infty}\left(1-\psi(\rho)\right)\widetilde{\psi}_j(\rho) \frac{\sin(R\rho)}{\rho}\cos(\rho\sqrt{\CL})f(z,t) \, d\rho.$$ 

In fact, we can handle 
$$\int_{0}^{\infty}\left(1-\psi(\rho)\right)\Psi(\rho) \frac{\sin(R\rho)}{\rho}\cos(\rho\sqrt{\CL})f(z,t) \, d\rho$$ 
in a similar manner since $\left(1-\psi(\rho)\right)\Psi(\rho)$ is compactly supported away from $0$. 

For each $j \geq 1,$ we define 
$$m_{R,j}(\eta) = \int_0^\infty\left(1-\psi(\rho)\right) \widetilde{\psi}_j(\rho) \frac{\sin(R\rho)}{\rho}\cos(\rho\eta) \, d\rho.$$

Using the above expression for $m_{R,j}$, we can rewrite $B_{R,j} f(z,t)$ as 
\begin{align*} 
B_{R,j} f(z,t) &= m_{R,j}(\sqrt{\CL})f (z,t) \\
&= c_d \int_{\R\setminus\{0\}}\sum_{k\geq 0}m_{R,j}\left(\sqrt{(2k+d)|\lambda|}\right) P_{k,\lambda}(f^{-\lambda})(z) e^{i\lambda t}|\lambda|^d \, d\lambda.
\end{align*}

We now study the following weighted bound estimate for $\sum_{j \geq 1} B_{R,j}f$. 

\begin{thm}\label{A1} Let $K\subset \H^d$ be compact.
For any $q<1$, and $0<\eta<1$, the following estimate  holds: 
\be \label{mainestimate} 
\int_{\|(z,t)\| < 2 A^{-2}} \sup_{R>1} R^{q} \left|\sum_{j \geq 1} B_{R,j} f(z,t) \right|^2 \, dz \, dt \lesssim_{K,q, \eta} \int_{\|(z,t)\|\geq 3}\frac{|f(z,t)|^2 }{\|(z,t)\|^{\eta}} \, dz \, dt,\ee 
for any $f\in L^2(\H^d \setminus \{\|(z,t)\|<3\}, \|(z,t)\|^{-\eta} \, dz \, dt)$ such that $\text{supp}(f)\subset K.$ 
\end{thm}

Using standard density arguments, Theorem \ref{main} follows immediately from Theorem \ref{A1}. In order to prove Theorem \ref{A1}, we need to first estimate $m_{R,j}$'s and their derivatives. 
\begin{rem}
In \cite{CS}, the authors proved an estimate similar to that in Theorem \ref{A1} above for $q=0$ without any assumption on the support of $f$ (See Theorem $2.2$ in \cite{CS}). However, for compactly supported functions, the inequality \eqref{mainestimate}  allows us to put a factor of $R^q$, for any $q<1$, in front of $S_Rf$ and we still get an a.e. localisation principle. If we restrict only to compactly supported functions in Theorem $2.2$ \cite{CS}, we would get a similar inequality as \eqref{mainestimate} above.
\end{rem}

\begin{lem}\label{A}
For any $l, k \in \N \cup \{0\}$, there exists a constant $C_{l,k} > 0$ such that 
$$\left|\frac{d^l}{dR^l}m_{R,j}(\eta)\right| \leq C_{l,k} \frac{2^{jl}}{(1+2^j| R-\eta|)^k},$$
for all $\eta\geq 0$ and $j, R\geq 1.$ 
\end{lem}
\begin{proof}
Recall that  $$m_{R,j}(\eta) = \int_0^\infty\left(1-\psi(\rho)\right) \widetilde{\psi}_j(\rho)  \frac{\sin (R\rho)}{\rho}\cos(\rho\eta)\, d\rho.$$

Now, since $1-\psi(\rho)\equiv 1$ on $[\epsilon_0,\infty)$, we have $\left(1-\psi(\rho)\right) \widetilde{\psi}_j(\rho) = \widetilde{\psi}_j(\rho)$, for every $j\geq 1$ (for sufficiently small $\epsilon_0 >0$).
A further simplification gives us that 
\begin{align*} 
m_{R,j}(\eta) &=  \frac{1}{2}\int_0^\infty \widetilde{\psi}_j(\rho) \frac{\sin \left((R+\eta)\rho\right) + \sin \left((R-\eta)\rho\right) }{\rho}\, d\rho \\
&= \frac{1}{2} \int_0^\infty \left(\Psi(\rho) - \Psi(2\rho)\right) \frac{\sin \left((R+\eta) 2^j\rho\right) + \sin \left((R-\eta)2^j\rho\right) }{\rho}\, d\rho. 
\end{align*} 

Recalling that $\Psi(\rho) - \Psi(2\rho)$ is supported in $[\frac{1}{2},2]$, the usual integration by parts gives us the desired  estimates for $m_{R,j}$ and it's derivatives.
\end{proof}
%

We next move to a key estimate, which in essence is the analogue of Lemma 2.3 of \cite{CS}, and this will be used further in the proof of Theorem \ref{A1}. For this, we define the operator $T_{\epsilon}$ as 
$$T_{\epsilon} f(z,t) = \int_{\R\setminus\{0\}}  \sum_{k \geq 0} \chi_{\{|1-\sqrt{|\lambda|(2k+d)}| < \epsilon\}}((2k+d)|\lambda|) P_{k,\lambda}(f^{-\lambda})(z) e^{i\lambda t}|\lambda|^d \, d\lambda,$$ 
for suitable functions. We have the following estimate for the operator $T_{\epsilon}$. 

\begin{restatable}{lem}{technicallemma}\label{B}
For any $0 \leq \beta <1$, and $\epsilon>0$, we have 
\begin{equation}\label{key-ineq}
\int_{\H^d} \left|T_{\epsilon}f(z,t) \right|^2 \, dz \, dt \lesssim_{\beta} 
\begin{cases}
\epsilon^{\beta} \int_{\H^d} |f(z,t)|^2 \|(z,t)\|^{2\beta} \, dz \, dt & \textup{when } 0<\epsilon<1;\\ \\
\epsilon^{4\beta} \int_{\H^d} |f(z,t)|^2 \|(z,t)\|^{2\beta} \, dz \, dt & \textup{when } \epsilon \geq 1. 
\end{cases}
\end{equation}
for any $f \in L^2(\mathbb{H}^d, \|(z,t)\|^{2\beta} \, dz \, dt)$. Moreover, $\epsilon^{\beta}\|(z,t)\|^{2\beta}$  on the R.H.S of the inequality above cannot, in general, be replaced by $\epsilon^{\beta}\|(z,t)\|^{\gamma}$ for any $\gamma < 2\beta.$ 
\end{restatable}

We postpone the proof of Lemma \ref{B} to the next section. 

\begin{rem} If we compare the above inequality with the one in Lemma $2.3$ in \cite{CS} (for the Laplacian $\Delta$ on $\R^d$), there it was shown that 
$$\int_{||\xi|-t| \leq \delta}|h(\xi)|^2 d\xi\leq c_{\alpha}\delta^{2\alpha}\int_{\R^n}|\hat{h}(\xi)|^2 |\xi|^{2\alpha}d\xi,$$ 
whenever $0\leq \alpha<1/2$, $t> 0 $ and $0 < \delta < 2t$. Here $\widehat{h}$ is the Euclidean Fourier transform of $h$ on $\R^n$. The powers of $\delta$ and the weight function $|\xi|$ on the R.H.S. of the inequality above are the same. 
However, in Lemma \ref{B} above, because of the non-isotropic homogeneity of the Cygan-Kor\'anyi norm on $\H^n$, we do not get the same power in $\epsilon$ and the weight function $\|(z,t)\|$.
\end{rem}	

\begin{rem}\label{refine-rem} As a consequence of Lemma \ref{B}, we get a refinement when $f$ is a compactly supported function. More precisely, given a compact set $K\subset \H^d$, for any $0 \leq \beta<1$ and $\eta<\beta,$ 
\begin{equation}\label{refine}
\int_{\H^d} \left|T_{\epsilon}f(z,t) \right|^2 \, dz \, dt \lesssim_{K,\beta, \eta} 
\begin{cases}
\epsilon^{\beta} \int_{\H^d} |f(z,t)|^2 \|(z,t)\|^{2\eta} \, dz \, dt & \textup{when } 0<\epsilon<1;\\ \\
\epsilon^{4\beta} \int_{\H^d} |f(z,t)|^2 \|(z,t)\|^{2\eta} \, dz \, dt & \textup{when } \epsilon \geq 1. 
\end{cases} 
\end{equation}
for any $f \in L^2(\mathbb{H}^d, \|(z,t)\|^{2\eta} \, dz \, dt)$ such that $\text{supp}(f)\subset K.$ 
\end{rem}

We are now in a position to prove Theorem \ref{A1}. 

\begin{proof}[{\bf Proof of Theorem \ref{A1}}]
Fix a $\beta \in (0,1)$ and $0<\eta<\beta$. Define $2\mu=\beta-\eta.$ Let $\omega\in C^{\infty}(\R)$ be such that $0 \leq \omega \leq 1$ on $\R,$ $\omega(R)\equiv 1$ when $R\geq 1$ and $\omega(R) \equiv 0$ when $R\leq \frac{1}{2}.$ For convenience, we work with the operator $\omega(R) B_{R,j}.$ Now,

\begin{align} \label{sobolev-ineq}
\sup_{R>1} R^{2\mu} \left|\sum_{j\geq 1} B_{R,j} f(z,t) \right|^2  & \leq \sup_{R \in \R} R^{2\mu}  \left|\omega(R)\sum_{j\geq 1} B_{R,j} f(z,t) \right|^2 \\
\nonumber & \lesssim_\gamma \int_{\R} \left| \sum_{j\geq 1} D^{\gamma} \left(R^{\mu}\omega(R) B_{R,j}f(z,t)\right) \right|^2 \, dR \\
\nonumber & \lesssim_\gamma \left(\sum_{j\geq 1} \left(\int_{\R} \left| D^{\gamma} \left(R^{\mu} \omega(R)B_{R,j}f(z,t)\right) \right|^2 \, dR\right)^{1/2}\right)^2,
\end{align}
for any $\frac{1}{2} < \gamma < 1,$ where the second last estimate follows from Sobolev inequality, with $D^{\gamma}$ denoting the fractional derivative in $R$-variable defined by the Euclidean Fourier multiplier as \begin{align}\label{def-Fourier-mult}
\widehat{D^{\gamma} g}(\tau)= |\tau|^{\gamma} \widehat{g}(\tau).
\end{align}

Therefore, by Minkowski's inequality 
\begin{align}\label{infinite-sum} 
& \left(\int_{\|(z,t)\| < 2 A^{-2}} \sup_{R>1} R^{2\mu} \left| \omega(R) \sum_{j\geq 1} B_{R,j} f(z,t) \right|^2 \, dz \, dt\right)^{1/2} \\
\nonumber \begin{split}
&\quad \quad \lesssim_\gamma \sum_{j\geq 1} \left( \int_{\|(z,t)\| < 2 A^{-2}} \int_{\R} \left| D^{\gamma} \left(R^{\mu} \omega(R)B_{R,j}f(z,t)\right) \right|^2 \, dR \, dz \, dt \right)^{1/2}.
\end{split} 
\end{align}

We will now analyse terms corresponding to each index $j$ separately. Let us first consider the case $\gamma=0$. Note that 
$$B_{R,j}f (z,t) = \left(f \ast K_{R,j} \right) (z,t),$$
where 
$$K_{R,j}(z,t) =  \int_0^\infty (1-\psi(\rho)) \widetilde{\psi}_j(\rho)  \frac{ \sin(R\rho)}{\rho} \cos(\rho \sqrt{\CL}) \delta (z,t) \, d\rho,$$
where $\delta$ denotes the Dirac distribution at origin in $\H^d$. We claim that $K_{R,j}$ is compactly supported in $\{\|(z,t)\| \leq 2^{j+1} A\}.$ This claim follows from the fact that $\cos(\rho \sqrt{\CL}) \delta$ is compactly supported in $\{\|(z,t)\| \leq \rho A\}$, because of the finite speed of propagation of the wave equation (as discussed in the beginning of this section), and $\widetilde{\psi}_j (\rho)$ is supported in the interval $[2^{j-1}, 2^{j+1}]$. For a large enough $c>0$, it follows directly from the definition of convolution on $\H^d$ that 
\begin{align*} 
\left(\chi_{\{\| \cdot \| > c 2^{j}\}} f \right) \ast K_{R,j} (z,t) = 0,
\end{align*}
for all $(z,t)$ with $\|(z,t)\| < 2 A^{-2}$ and for all $j$. As a consequence, we may assume without loss of generality that $f$ is supported in the ball $\|(z,t)\| \leq c 2^{j}$, and we shall shortly make use of this assumption. 

Now, as also done in \cite{CS}, it is enough to show that 
\begin{align}\label{compare-eq-1} 
&\int_{\H^d} \int_0^{\infty} R^{2\mu} \omega(R)^2 |f \ast K_{R,j}(z,t)|^2 {\|(z,t)\|^{-\eta}} \, dz \, dt \, dR\\
\nonumber \begin{split}
&\quad \quad \lesssim_{\beta,\eta} 2^{-j\beta} 2^{-j} \int_{\H^d}|f(z,t)|^2 \, dz \, dt.
\end{split} 
\end{align}

By duality, the above is equivalent to proving 
\begin{align}\label{compare-eq-2} 
& \lefteqn{\int_{\H^d}\left|\int_0^{\infty} f_{R} \ast K_{R,j} (z,t) \omega(R) \, dR\right|^2 \, dz \, dt} \\
\nonumber \begin{split}
&\quad \quad \lesssim_{\beta,\eta} 2^{-j\beta} 2^{-j} \int_{\H^d}\int_0^{\infty} \omega(R)^2|f_R(z,t)|^2 {\|(z,t)\|^{\eta}}R^{-2\mu} \, dR \, dz \, dt. 
\end{split} 
\end{align}

After applying Euclidean Plancherel theorem in the $t$-variable and orthogonality of the special Hermite projections, we can write
\begin{align*} 
& \lefteqn{\int_{\H^d}\left|\int_0^{\infty} f_{R} \ast K_{R,j}  (z,t) \omega(R) \, dR\right|^2 \, dz \, dt} \\
\begin{split}
&\quad = c_d \int_{\R^{2d+1}}\sum_k \left|\int_{0}^{\infty}m_{R,j}(\sqrt{(2k+d)|\lambda|}) P_{k, \lambda} (f_R^{\lambda})(z) |\lambda|^d \omega(R) \, dR\right|^2 \, d\lambda \, dz. 
\end{split} 
\end{align*}

By Lemma \ref{A}, it is easy to see that 
$$|m_{R,j}(\sqrt{(2k+d)|\lambda|})| \lesssim_L \sum_n 2^{-nL} \chi_{\{2^j|\sqrt{(2k+d)|\lambda|}-R|\leq 2^n\}}((2k+d)|\lambda|),$$ for $L>0$ large enough. So, it is enough to look at
$$\int_{\R^{2d+1}}\sum_k \left|\int_{0}^{\infty}\chi_{\{2^j|\sqrt{(2k+d)|\lambda|}-R|\leq 2^n\}}((2k+d)|\lambda|) P_{k, \lambda} (f_R^{\lambda})(z) \omega(R) \, dR\right|^2 |\lambda|^{2d} \, d\lambda \, dz.$$

Now, we apply Cauchy Schwarz inequality to get 
\begin{align*} 
& \lefteqn{\left|\int_{0}^{\infty}\chi_{\{2^j|\sqrt{(2k+d)|\lambda|}-R|\leq 2^n\}}((2k+d)|\lambda|) P_{k, \lambda} (f_R^{\lambda})(z) \omega(R) \, dR\right|^2} \\
\begin{split}
&\quad \quad \leq 2^{n-j}\int_{0}^{\infty}\chi_{\{2^j|\sqrt{(2k+d)|\lambda|}-R|\leq 2^n\}}((2k+d)|\lambda|) |P_{k, \lambda} (f_R^{\lambda})(z)|^2 \omega(R)^2 \, dR. 
\end{split} 
\end{align*}
We now consider the following two cases:

\medskip \noindent {\bf Case 1. When $\frac{2^{n-j}}{R}<1$.}

Using the non-isotropic dilation $(z,t)\rightarrow \delta_{R}(z,t)$ for a fixed $R>0$, it is easy to see that by putting $\epsilon=\frac{2^{n-j}}{R}$ in the inequality \eqref{refine} of Remark \ref{refine-rem} we get
\begin{align*}
& \int_{\R^{2d}}\int_{\R\setminus\{0\}}\sum_k\left|\int_{0}^{\infty}\chi_{\{2^j|\sqrt{(2k+d)|\lambda|}-R|\leq 2^n\}}((2k+d)|\lambda|) P_{k, \lambda} (f_R^{\lambda})(z) \omega(R) \, dR\right|^2 |\lambda|^{2d} \, d\lambda \, dz\\
\nonumber & \leq 2^{n-j}\int_{0}^{\infty} \int_{\R^{2d}}\int_{\R\setminus\{0\}} \chi_{\{2^j|\sqrt{(2k+d)|\lambda|}-R| \leq 2^n\}}((2k+d)|\lambda|) |P_{k, \lambda}(f_R^{\lambda})(z)|^2 \omega(R)^2 \, dz \, dt \, dR\\
\nonumber &\lesssim_{\beta, \eta} 2^{n-j}2^{\beta(n-j)}\int_{\H^d}\int_{0}^{\infty}\omega(R)^2|f_R(z,t)|^2 \|(z,t)\|^{\eta} R^{\eta-\beta} \, dz \, dt \, dR,
\end{align*}
for any $0 \leq \beta <1$ and $\eta<2\beta,$ which is clearly satisfied as we have chosen  $\eta<\beta$. 

\medskip \noindent {\bf Case 2. When $\frac{2^{n-j}}{R}\geq1$ }. 

Once again, we put $\epsilon=\frac{2^{n-j}}{R}$, and apply the non-isotropic dilation $(z,t)\rightarrow (R z,R^2 t)$ for a fixed $R>0$ in the inequality \eqref{refine} of Remark \ref{refine-rem}, and we have  
\begin{align*} 
& \int_{\R^{2d}}\int_{\R\setminus\{0\}}\sum_k\left|\int_{0}^{\infty}\chi_{\{2^j|\sqrt{(2k+d)|\lambda|}-R|\leq 2^n\}}((2k+d)|\lambda|) P_{k, \lambda} (f_R^{\lambda})(z) \omega(R) \, dR\right|^2 |\lambda|^{2d} \, d\lambda \, dz\\
\nonumber & \leq 2^{n-j}\int_{0}^{\infty} \int_{\R^{2d}}\int_{\R\setminus\{0\}} \chi_{\{2^j|\sqrt{(2k+d)|\lambda|}-R|\leq 2^n\}}((2k+d)|\lambda|) |P_{k, \lambda} (f_R^{\lambda})(z)|^2 \omega(R)^2 |\lambda|^{2d} \, d\lambda\, dz  \, dR\\
\nonumber & \lesssim_{\beta,\eta} 2^{n-j}2^{4\beta(n-j)}\int_{\H^d}\int_{0}^{\infty}\omega(R)^2|f_R(z,t)|^2 \|(z,t)\|^{\eta} R^{\eta-4\beta} \, dz \, dt \, dR \\
\nonumber & \lesssim_{\beta,\eta} 2^{n-j}2^{4\beta(n-j)}\int_{\H^d}\int_{0}^{\infty}\omega(R)^2|f_R(z,t)|^2 \|(z,t)\|^{\eta} R^{\eta-\beta} \, dz \, dt \, dR, 
\end{align*}
for any $0 \leq \beta <1$ and $\eta<\beta.$ The last inequality follows from the fact that $\omega$ is supported on $[1/2, \infty)$. We can choose $L$ large enough to make the above sum convergent in $n$ and take the maximum over weight function in $R$ and the constant in the inequality as well.
Thus, we have 
\begin{align*} 
& \lefteqn{\int_{\H^d}\int_0^{\infty}R^{2\mu} \omega(R)^2 |B_{R,j}f(z,t)|^2 \|(z,t)\|^{-\eta} \, dz \, dt \, dR}\\
\begin{split}
& \quad \quad \lesssim_{\beta,\eta} 2^{-j-\beta j} \int_{\H^d} |f(z,t)|^2 \, dz \, dt \\ 
& \quad \quad \sim_{\beta,\eta} 2^{-j-\beta j} \int_{\|(z,t)\| \leq c 2^{j}} |f(z,t)|^2 \, dz \, dt \\ 
& \quad \quad \lesssim_{\beta,\eta, \delta} 2^{-j-\delta j} \int_{\|(z,t)\| \leq c 2^{j}} |f(z,t)|^2 \|(z,t)\|^{-(\beta-\delta)} \, dz \, dt \\
& \quad \quad \lesssim_{\beta,\eta, \delta} 2^{-j-\delta j} \int_{\H^d} |f(z,t)|^2 \|(z,t)\|^{-(\beta-\delta)} \, dz \, dt 
\end{split} 
\end{align*}
for any $0 < \beta < 1,$ and $\delta>0$ such that $\beta-\delta > 0.$ In the above estimation, we made use of the fact that $f$ is supported in the ball $\|(z,t)\| \leq c 2^{j}$.  

When $\gamma=1,$ denoting by $D^1$ the distributional derivative in $R$-variable $\frac{d}{dR}$, one could use similar method as above to prove that
\begin{align*} 
& \lefteqn{\int_{\H^d}\int_{0}^{\infty} \left|\frac{d}{dR} (R^{\mu}\omega(R)B_{R,j}f(z,t))\right|^2 \|(z,t)\|^{-\eta} \, dz \, dt \, dR}\\
\begin{split}
&\quad \lesssim_{\beta,\eta, \delta} 2^{j-\delta j}\int_{\H^d} |f(z,t)|^2 \|(z,t)\|^{-(\beta-\delta)} \, dz \, dt. 
\end{split} 
\end{align*}

And then for $0<\gamma<1$, one could apply interpolation theorem for weighted $L^p$ spaces (see 5.3 of Section 5, Chapter 5 of \cite{SW} with $p_0 = p_1 =2$). In particular, we have 
\begin{align*} 
& \lefteqn{\int_{\H^d} \int_{\R} \left| D^{\gamma} \left(R^{\mu} \omega(R) B_{R,j}f(z,t)\right)\right|^2 \|(z,t)\|^{-\eta} \, dz \, dt \, dR} \\
\begin{split}
&\quad \lesssim_{\beta,\eta, \delta, \gamma} 2^{-j-\delta j +2\gamma j} \int_{\H^d} |f(z,t)|^2 \|(z,t)\|^{-(\beta-\delta)} \, dz \, dt. 
\end{split} 
\end{align*}

Finally, given $\delta>0$, one could choose $\frac{1}{2}<\gamma<1$ such that $1+\delta-2\gamma >0.$ Using these estimates in inequality \eqref{infinite-sum}, we have 
\begin{align*}
& \int_{\|(z,t)\| < 2 A^{-2}} \sup_{R>1}R^{2\mu} | \sum_{j\geq 1} B_{R,j} f(z,t)|^2 \, dz \, dt \\
\begin{split}
& \lesssim_{\beta, \delta, \gamma} \left(\sum_{j\geq 1} \left( \int_{\|(z,t)\| < 2 A^{-2}} \int_{\R} \left| D^{\gamma} \left(R^{\mu} \omega(R) B_{R,j}f(z,t)\right) \right|^2 \, dR \, dz \, dt \right)^{\frac{1}{2}}\right)^2 \\
& \lesssim_{\beta, \delta, \gamma} \left(\sum_{j\geq 1} \left( 2^{-j-\delta j +2\gamma j} \int_{\H^d}|f(z,t)|^2 \|(z,t)\|^{-(\beta-\delta)} \, dz \, dt \right)^{\frac{1}{2}} \right)^2\\
& \lesssim_{\beta, \delta, \gamma} \int_{\H^d}|f(z,t)|^2 \|(z,t)\|^{-(\beta-\delta)} \, dz \, dt.
\end{split} 
\end{align*}

Since $0 < \delta < \beta$ is arbitrary, this completes the proof of Theorem \ref{A1}.
\end{proof}

\section{Technical Lemmas}\label{key-estimates}
We will prove Lemma \ref{B} in this section. For convenience, we state it again. 

\technicallemma*
\begin{proof}
We would like to first remark that when $\epsilon$ is away from $0$, the estimate of the lemma holds true easily. To see this, let $\epsilon \geq 1$, then $\left|1-\sqrt{|\lambda|(2k+d)}\right| < \epsilon$ implies $|\lambda|(2k+d) < 4 \epsilon^2$. In this case one could use Plancherel theorem for the Euclidean space in last variable and the orthogonality of $P_{k,\lambda}$'s in $L^2(\C^d)$ to have 
\begin{align}\label{ineq-scaled-key-lemma}
\int_{\H^d} \left|T_{\epsilon}f(z,t) \right|^2 \, dz \, dt & \leq C_{\beta} \epsilon^{4\beta} \int_{\R\setminus\{0\}} \sum_{k \geq 0} \frac{1}{(|\lambda|(2k+d))^{2\beta}} \|P_{k,\lambda}(f^{-\lambda})\|^2_{L^2(\C^d)} |\lambda|^{2d} \, d\lambda \\
\nonumber & = C_{\beta} \epsilon^{4\beta} \int_{\H^d} |\CL^{-\beta}f(z,t)|^2 \, dz \, dt \\
\nonumber & \lesssim_{\beta} \epsilon^{4\beta} \int_{\H^d} |f(z,t)|^2 \|(z,t)\|^{2\beta} \, dz \, dt,
\end{align}
where the last inequality follows from the Hardy-Sobolev inequality on the Heisenberg group (see Section 3 in \cite{CCR}). 

Let us now assume that $0< \epsilon < 1$. Again, using the Plancherel theorem for the Euclidean space in the last variable and orthogonality of $P_{k,\lambda}$'s in $L^2(\C^d)$, we get 
\begin{align*} 
& \int_{\H^d} \left|T_{\epsilon}f(z,t) \right|^2 \, dz \, dt \\ 
&\quad = c_d \int_{\R\setminus\{0\}} \sum_{k \geq 0} \chi_{\{|1-\sqrt{|\lambda|(2k+d)}| < \epsilon\}} ((2k+d)|\lambda|) \|P_{k,\lambda}(f^{-\lambda})\|^2_{L^2(\C^d)} |\lambda|^{2d} \, d\lambda \\ 
&\quad \lesssim \int_{\R\setminus\{0\}} \sum_{k \geq 0} \chi_{\{|1-|\lambda|(2k+d)| < 3\epsilon\}} ((2k+d)|\lambda|) \|P_{k,\lambda}(f^{-\lambda})\|^2_{L^2(\C^d)} |\lambda|^{2d} \, d\lambda \\ 
&\quad = \int_{\R\setminus\{0\}} \sum_{\frac{1-3\epsilon}{|\lambda|} <2 k+d < \frac{1+3\epsilon}{|\lambda|} } \|P_{k,\lambda}(f^{-\lambda})\|^2_{L^2(\C^d)} |\lambda|^{2d} \, d\lambda \\ 
&\quad = I_{1,\epsilon} + I_{2,\epsilon},
\end{align*}
where, 
\begin{align*}
I_{1,\epsilon} & := \int_{\R\setminus\{0\}} \sum_{\substack{\frac{1-3\epsilon}{|\lambda|} <2 k+d < \frac{1+3\epsilon}{|\lambda|} \\ \lfloor 6 \epsilon/ |\lambda|\rfloor \geq 1}} \|P_{k,\lambda}(f^{-\lambda})\|^2_{L^2(\C^d)} |\lambda|^{2d} \, d\lambda, \\
I_{2,\epsilon} & := \int_{\R\setminus\{0\}} \sum_{\substack{\frac{1-3\epsilon}{|\lambda|} <2 k+d < \frac{1+3\epsilon}{|\lambda|} \\ \lfloor 6\epsilon/ |\lambda|\rfloor =0}} \|P_{k,\lambda}(f^{-\lambda})\|^2_{L^2(\C^d)} |\lambda|^{2d} \, d\lambda. 
\end{align*}

Now, using the fact that 
$$\sum_{k \geq 0} \|P_{k,\lambda}(f^{-\lambda})\|^2_{L^2(\C^d)} |\lambda|^{2d} = C_d \|f^{\lambda}\|^2_{L^2(\C^d)},$$
we get 
$$I_{1,\epsilon} \lesssim_{\beta} \epsilon^{\beta} \int_{\R\setminus\{0\}} \frac{1}{|\lambda|^{\beta}} \|f^{\lambda}\|^2_{L^2(\C^d)} \, d\lambda,$$
for every $\beta \geq 0$. Restricting $\beta$ in the open interval $(0,1)$, we invoke Hardy's inequality for fractional derivatives in $\lambda$-variable to have 
\begin{align}\label{Hardy-fractional}
\int_{\R\setminus\{0\}} \frac{1}{|\lambda|^{\beta}} \left| f^{\lambda} (z) \right|^2 \, d\lambda \lesssim_{\beta} \int_{\R} \left| D^{\beta/2} f^{\lambda} (z) \right|^2 \, d\lambda = C_\beta \int_{\R} |f(z,t)|^2 |t|^{\beta} \, dt,
\end{align}
where $D^{\beta/2}$ denotes the fractional derivative in $\lambda$-variable (as in \eqref{def-Fourier-mult}). 

Therefore, for any $\beta \in [0,1)$, 
\Bea 
I_{1,\epsilon} &\lesssim_{\beta}& \epsilon^{\beta} \int_{\H^d} |f(z,t)|^2 |t|^{\beta} \, dz \, dt.
\Eea

On the other hand, using relation \eqref{e_k}, we have that
\begin{align*}
I_{2,\epsilon} 
&= \int_{\R\setminus\{0\}} \sum_{\substack{\frac{1-3\epsilon}{|\lambda|} <2 k+d < \frac{1+3\epsilon}{|\lambda|} \\ \lfloor 6\epsilon/ |\lambda|\rfloor =0}} \|P_{k,\lambda}(f^{-\lambda})\|^2_{L^2(\C^d)} |\lambda|^{2d} \, d\lambda \\
&= c_d \int_{\C^d} \int_{\R\setminus\{0\}} \sum_{\substack{\frac{1-3\epsilon}{|\lambda|} <2 k+d < \frac{1+3\epsilon}{|\lambda|} \\ \lfloor 6\epsilon/ |\lambda|\rfloor =0}} |(f \ast e_k)^{-\lambda}(z)|^2 \, d\lambda \, dz \\
&\lesssim \int_{\C^d} \sum_{2 k+d \leq \frac{1}{3\epsilon}} \int_{\frac{1-3\epsilon}{2k+d}\leq |\lambda|\leq\frac{1+3\epsilon}{2k+d}} |(f \ast e_k)^{-\lambda}(z)|^2 \, d\lambda \, dz.
\end{align*}

For $0 \leq \beta <1$, using Lemma 2.3 of \cite{CS} in the above estimate, one gets 
\begin{align*}
I_{2,\epsilon} &\lesssim_\beta \epsilon^{\beta} \sum_{2 k+d \leq \frac{1}{3\epsilon}} \int_{\H^d} |(f \ast e_k)(z,t)|^2 (2k+d)^{-\beta} |t|^{\beta} \, dz \, dt \\
&\lesssim_\beta \epsilon^{\beta} \sum_{k \geq 0} \int_{\H^d} |(f \ast e_k)(z,t)|^2 (2k+d)^{-\beta} |t|^{\beta} \, dz \, dt, 
\end{align*}

If we could show that for any $0 \leq \beta <1,$
\begin{equation} \label{key-interpolation-ineq}
\sum_{k \geq 0} \int_{\H^d} |(f \ast e_k)(z,t)|^2 (2k+d)^{-\beta} |t|^{\beta} \, dz \, dt \lesssim_{\beta} \int_{\H^d} |f(z,t)|^2 \|(z,t)\|^{2\beta} \, dz \, dt, 
\end{equation}
then that would complete the proof of inequality \eqref{key-ineq} as stated in Lemma \ref{B}. 

In fact, the following estimate holds true 
\begin{align}\label{interpolated-eq}
&\quad \quad \lefteqn{\sum_{k \geq 0} \int_{\H^d} |(f \ast e_k)(z,t)|^2 (2k+d)^{-\beta} |t|^{\beta} \, dz \, dt} \\
\nonumber &\lesssim_{\beta} \int_{\H^d} |f(z,t)|^2 |t|^{\beta} \, dz \, dt + \int_{\C^d} \int_{\R\setminus\{0\}} \frac{\left| f^{\lambda} (z) \right|^2}{|\lambda|^{\beta}} \, d\lambda \, dz + \int_{\C^d} \int_{\R\setminus\{0\}} \frac{|z|^\beta \left| f^{\lambda} (z) \right|^2}{|\lambda|^{\beta /2}} \, d\lambda \, dz,
\end{align}
for any $0 \leq \beta \leq 2$. We shall prove inequality \eqref{interpolated-eq} later in this section (Lemma \ref{A4}). 

Assuming inequality \eqref{interpolated-eq} for now, and by restricting $\beta$ in the open interval $(0,1)$, we can once again invoke Hardy's inequality for fractional derivatives in $\lambda$-variable (see \eqref{Hardy-fractional}). Thus, for $0 \leq \beta < 1$, inequality \eqref{interpolated-eq} implies
\begin{align*}
\sum_{k \geq 0} \int_{\H^d} |(f \ast e_k)(z,t)|^2 (2k+d)^{-\beta} |t|^{\beta} \, dz \, dt 
& \lesssim_{\beta} \int_{\H^d} |f(z,t)|^2 (|z|^{\beta}|t|^{\beta/2}) + |t|^{\beta}) \, dz \, dt \\
& \lesssim_{\beta} \int_{\H^d} |f(z,t)|^2 \|(z,t)\|^{2\beta} \, dz \, dt.
\end{align*}
This completes the proof of the inequality \eqref{key-interpolation-ineq} and hence of \eqref{key-ineq}. 

Finally, we shall show that $\epsilon^{\beta}\|(z,t)\|^{2\beta}$ which appears on the right hand side of inequality \eqref{key-ineq} in the statement of Lemma \ref{B} cannot, in general, be replaced by $\epsilon^{\beta}\|(z,t)\|^{\gamma}$ for any $\gamma < 2\beta.$ To show this, let us assume on contrary that there exists some $\gamma < 2\beta$ such that for all $0< \epsilon<1,$ we have  \begin{equation}\label{key-ineq-false} 
\int_{\H^d} \left|T_{\epsilon}f(z,t) \right|^2 \, dz \, dt \lesssim_{\beta, \gamma} \epsilon^{\beta} \int_{\H^d} |f(z,t)|^2 \|(z,t)\|^{\gamma} \, dz \, dt, 
\end{equation} 
for all $f$ such that the right hand side of the above inequality is finite.

Consider now a function $f$ such that  $f^{\lambda}(z) = \widehat{g}(\lambda) \chi_{B(0,\eta)}(z),$ where $g$ is a Schwartz class function on $\R,$ $\eta$ a fixed positive real number, and $B(0,\eta)$ the open ball in $\C^d$ centred at the origin and of radius $\eta$.

By definition of the operator $T_\epsilon$, for any $0<\epsilon<1$, we have 
\begin{align*} 
& \int_{\H^d} \left|T_{\epsilon} f(z,t) \right|^2 \, dz \, dt \\ 
&\quad =  \int_{\R\setminus\{0\}} \sum_{k \geq 0} \chi_{\{|1-\sqrt{|\lambda|(2k+d)}| < \epsilon\}} ((2k+d)|\lambda|) \|P_k^\lambda(f^{-\lambda})\|^2_{L^2(\C^d)} |\lambda|^{2d} \, d\lambda \\ 
&\quad \geq \int_{\R\setminus\{0\}} \sum_{k \geq 0} \chi_{\{|1-|\lambda|(2k+d)| < \epsilon\}} ((2k+d)|\lambda|) \|P_k^\lambda(f^{-\lambda})\|^2_{L^2(\C^d)} |\lambda|^{2d} \, d\lambda \\ 
&\quad \geq  \int_{\R\setminus\{0\}} \chi_{\{|1- d |\lambda|| < \epsilon\}} ((2k+d)|\lambda|) \|P_0^\lambda(f^{-\lambda})\|^2_{L^2(\C^d)} |\lambda|^{2d} \, d\lambda \\ 
&\quad = C \int_{|1- d |\lambda|| < \epsilon} \left|\int_{\C^d} f^{-\lambda}(z)e^{\frac{-|\lambda||z|^2}{2}}dz\right|^2|\lambda|^d \, d\lambda.
\end{align*}

In particular, for our choice of $f$, the following holds: 
\begin{align*}
& \lefteqn{\int_{|1- d |\lambda|| < \epsilon} \left|\int_{\C^d} \chi_{B(0,\eta)}(z) e^{\frac{-|\lambda||z|^2}{2}} \, dz \right|^2 |\widehat{g}(-\lambda)|^2|\lambda|^d \, d\lambda} \\
\begin{split}
&\quad\quad \lesssim_{\beta} \epsilon^{\beta} \int_{\H^d} |g(t)|^2 \chi_{B(0,\eta)}(z) \|(z,t)\|^{\gamma} \, dz \, dt.
\end{split} 
\end{align*}

For $\epsilon>0$ small, $|\lambda|\sim d^{-1}$ in the integral on the L.H.S. of the above inequality. Using this fact, we get that
$$\int_{|1- d |\lambda|| < \epsilon}| \widehat{g}(\lambda)|^2 \, d\lambda \lesssim_{\beta, \eta} \epsilon^{\beta} \left(\eta^{\gamma} \int_{|t|\leq \eta^2} |g(t)|^2 \, dt + \int_{|t|\geq \eta^2} |g(t)|^2 |t|^{\frac{\gamma}{2}} \, dt \right).$$

By change of variables, it is easy to see that for $R>1,$ the last inequality implies that 
\begin{align} \label{key-ineq-false-special-choice} 
& \lefteqn{\int_{||\lambda|-R|<\epsilon}| \widehat{g}(\lambda)|^2 \, d\lambda} \\
\nonumber \begin{split}
&\quad \lesssim_{\beta, \eta} \frac{\epsilon^{\beta}}{R^{\beta}} \left(\eta^{\gamma} \int_{|t|\leq \frac{\eta^2}{R}} |g(t)|^2 \, dt + R^{\frac{\gamma}{2}}\int_{|t|\geq \frac{\eta^2}{R}} |g(t)|^2 |t|^{\frac{\gamma}{2}} \, dt \right).
\end{split} 
\end{align}

Let us now choose $g$ such that $\widehat{g}(\lambda)= \phi(\lambda-R)$, for some $\phi \in C_c^{\infty}(\R)$ with $\phi \equiv 1$ on $|\lambda| \leq \frac{1}{2}.$ 
It is easy to see that $g(t)= e^{itR} \widehat{\phi}(-t)$. For the above choice of $g$, it is straightforward to verify that 
$$\int_{||\lambda|-R|<\epsilon}|\widehat{g}(\lambda)|^2 \, d\lambda \geq \int_{|\lambda -R|<\epsilon}|\widehat{g}(\lambda)|^2 \, d\lambda = 2 \epsilon.$$

Therefore, inequality \eqref{key-ineq-false-special-choice} implies that
$$\epsilon\lesssim_{\beta,\eta} \frac{\epsilon^{\beta}}{R^{\beta}} \left(\eta^{\gamma} \int_{|t|\leq \frac{\eta^2}{R}} |\widehat{\phi}(-t)|^2 \, dt + R^{\frac{\gamma}{2}}\int_{|t|\geq \frac{\eta^2}{R}} |\widehat{\phi}(-t))|^2 |t|^{\frac{\gamma}{2}} \, dt \right).$$
As we take $R\rightarrow\infty,$ the right hand side of the above inequality goes to $0$ because by assumption $\frac{\gamma}{2}< \beta.$ But $\epsilon>0,$ and therefore we arrive at a contradiction. This proves the sharpness of the exponent of the weight in the inequality \eqref{key-ineq} as stated in Lemma \ref{B}, and completes the proof of Lemma \ref{B}.
\end{proof}

We had claimed estimate \eqref{interpolated-eq} in the proof of Lemma \ref{B}, and we establish it below. 
\begin{lem}\label{A4}
For any $0 \leq \beta \leq 2$, we have 
\begin{align}\tag{\ref{interpolated-eq}}
&\quad \quad \lefteqn{\sum_{k \geq 0} \int_{\H^d} |(f \ast e_k)(z,t)|^2 (2k+d)^{-\beta} |t|^{\beta} \, dz \, dt} \\
\nonumber &\lesssim_{\beta} \int_{\H^d} |f(z,t)|^2 |t|^{\beta} \, dz \, dt + \int_{\C^d} \int_{\R\setminus\{0\}} \frac{\left| f^{\lambda} (z) \right|^2}{|\lambda|^{\beta}} \, d\lambda \, dz + \int_{\C^d} \int_{\R\setminus\{0\}} \frac{|z|^\beta \left| f^{\lambda} (z) \right|^2}{|\lambda|^{\beta /2}} \, d\lambda \, dz.
\end{align}

for functions on $\mathbb{H}^d$ for which the right hand side is finite. 
\end{lem}
\begin{proof}
Note that for $\beta = 0$, estimate \eqref{interpolated-eq} reduces to 
\be\label{interpolation-eq1}
\sum_{k \geq 0} \int_{\H^d} |(f \ast e_k)(z,t)|^2 \, dz \, dt \lesssim \int_{\H^d} |f(z,t)|^2 \, dz \, dt. 
\ee 
which is true in view of the Plancherel theorem for the Heisenberg group, and in fact the above estimate holds with equality (upto a constant multiple).

For $\beta = 2$, we will prove the estimate \eqref{interpolated-eq} by analysing each summand of the left hand side. For this, let us consider the functions $\omega$ and $\zeta_j$ on $\H^d$ defined by $\omega(z,t) = t$ and $\zeta_j(z,t)= z_j$. Note that 
\begin{align*} 
& \omega(z,t) (f \ast e_k)(z,t) \\ 
&= t \int_{\H^d} f(z-w, t-s-\frac{1}{2} Im(z \cdot \bar{w})) e_k(w,s) \, dw \, ds \\
&= \int_{\H^d} (t-s-\frac{1}{2} Im(z \cdot \bar{w})) f(z-w, t-s-\frac{1}{2} Im(z \cdot \bar{w})) e_k(w,s) \, dw \, ds \\ 
&\quad + \int_{\H^d} s f(z-w, t-s-\frac{1}{2} Im(z \cdot \bar{w})) e_k(w,s) \, dw \, ds \\ 
&\quad + \frac{1}{2} \int_{\H^d} Im(z \cdot \bar{w}) f(z-w, t-s-\frac{1}{2} Im(z \cdot \bar{w})) e_k(w,s) \, dw \, ds \\ 
& = (\omega f) \ast e_k (z,t) + f \ast (\omega e_k)  (z,t) + \frac{1}{2} \int_{\H^d} Im(z \cdot \bar{w}) f((z,t)(w,s)^{-1}) e_k(w,s) \, dw \, ds. 
\end{align*}

Plugging $Im(z \cdot \bar{w}) = - \frac{i}{2} \sum_{j=1}^d \left((z_j-w_j) \bar{w_j} - (\bar{z}_j-\bar{w}_j) w_j \right)$ in the last integral of the above expression, we get the following identity: 
\be 
\label{omega} \omega (f \ast e_k) = (\omega f) \ast e_k + f \ast (\omega e_k) -\frac{i}{4} \sum_{j=1}^d \left( (\zeta_j f) \ast (\bar{\zeta_j} e_k) -(\bar{\zeta_j} f) \ast (\zeta_j e_k) \right), 
\ee

Note that the term $\frac{i}{4} \sum_{j=1}^d \left( (\zeta_j f) \ast (\bar{\zeta_j} e_k) -(\bar{\zeta_j} f) \ast (\zeta_j e_k) \right)$ appears in the identity \eqref{omega} above due to the non-commutative group structure of $\H^d.$

Therefore, 
\begin{align} \label{Leibniz-Hn-ineq}
& \int_{\H^d} |(f \ast e_k)(z,t)|^2 (2k+d)^{-2} |t|^{2} \, dz \, dt \\
\nonumber & \quad  \lesssim \int_{\H^d} |((\omega f) \ast e_k)(z,t)|^2 (2k+d)^{-2} \, dz \, dt \\
\nonumber & \quad \quad + \int_{\H^d} |(f \ast (\omega e_k))(z,t)|^2 (2k+d)^{-2} \, dz \, dt \\
\nonumber & \quad  \quad + \sum_{j=1}^d \int_{\H^d} | (\zeta_j f) \ast (\bar{\zeta_j} e_k)(z,t)|^2  (2k+d)^{-2} \,dz \,dt \\
\nonumber & \quad \quad + \sum_{j=1}^d \int_{\H^d} | (\bar{\zeta_j} f) \ast (\zeta_j e_k)(z,t)|^2  (2k+d)^{-2} \,dz \,dt.
\end{align}

It is easy to handle the sum (over $k \geq 0$) of the first term on the right hand side of inequality \eqref{Leibniz-Hn-ineq}. In fact, 
\begin{align}\label{Leibniz-Hn-ineq1}
\sum_{k \geq 0} \int_{\H^d} |((\omega f) \ast e_k)(z,t)|^2 (2k+d)^{-2} \, dz \, dt & \leq \sum_{k \geq 0} \int_{\H^d} |((\omega f) \ast e_k)(z,t)|^2 \, dz \, dt \\
\nonumber & = c_d \int_{\H^d} |(\omega f) (z,t)|^2 \, dz \, dt \\
\nonumber & = c_d \int_{\H^d} |f(z,t)|^2 |t|^2 \, dz \, dt. 
\end{align}

To estimate the sum (over $k \geq 0$) of the second term of inequality \eqref{Leibniz-Hn-ineq}, we make use of the recurrence identities for Laguerre polynomials and Laguerre functions (see Page 92 of \cite{ST}, and equation (5.1.10) on Page 101 of \cite{SZ}) to verify that $\frac{d}{d\lambda} \left( \varphi_k (\sqrt{|\lambda|}z)|\lambda|^d \right)$ equals to 
$$\left(\frac{d}{2}\varphi_{k} (\sqrt{|\lambda|}z) - \frac{k+d-1}{2} \varphi_{k-1}(\sqrt{|\lambda|}z) + \frac{k+1}{2} \varphi_{k+1}(\sqrt{|\lambda|}z)\right) \frac{\lambda}{|\lambda|}|\lambda|^{d -1}.$$

Therefore when $z\neq 0,$ we have
\begin{align*} 
& (\omega e_k)(z,t) = t \int_{\R} e^{i\lambda t} \varphi_k (\sqrt{|\lambda|}z) |\lambda|^d \, d\lambda \\
& \quad = i \int_{\R} e^{i\lambda t} \frac{d}{d\lambda} \left( \varphi_k (\sqrt{|\lambda|}z) |\lambda|^d \right) \, d\lambda \\
& \quad = i \int_{\R\setminus\{0\}} \left( \frac{d}{2}\varphi_{k}(\sqrt{|\lambda|}z) - \frac{k+d-1}{2} \varphi_{k-1}(\sqrt{|\lambda|}z) + \frac{k+1}{2} \varphi_{k+1}(\sqrt{|\lambda|}z)\right) \frac{\lambda}{|\lambda|} \\ 
& \quad \quad \quad \quad \quad |\lambda|^{d -1} e^{i\lambda t} d\lambda.
\end{align*}
 
Then, using the Euclidean Plancherel theorem in $t$-variable, we have that 
$$\lefteqn{\int_{\R} |(f \ast (\omega e_k))(z,t)|^2 \, dt}$$ 
is equal to a constant multiple of 
\begin{align*} 
& \int_{\R\setminus\{0\}} \left| f^{-\lambda} \times_{\lambda} \left( \frac{d}{2}\varphi_{k}(\sqrt{|\lambda|}\cdot) - \frac{k+d-1}{2} \varphi_{k-1}(\sqrt{|\lambda|}\cdot) + \frac{k+1}{2} \varphi_{k+1}(\sqrt{|\lambda|}\cdot)\right) (z) \right|^2  \\ 
& \quad \quad |\lambda|^{2d-2} \, d\lambda.
\end{align*} 

Using \eqref{laguerre-expansion} and \eqref{twisted}, 
the above identity implies that $$\lefteqn{\int_{\R} |(f \ast (\omega e_k))(z,t)|^2 \, dt}$$ is equal to a constant multiple of 
\begin{align*}
& \int_{\C^d} \int_{\R\setminus\{0\}} \frac{d^2}{4|\lambda|^2} \left| f^{-\lambda} \times_{\lambda} \varphi_{k}(\sqrt{|\lambda|}\cdot)(z) \right|^2 |\lambda|^{2d} \, d\lambda \, dz \\
&+ \int_{\C^d} \int_{\R\setminus\{0\}} \frac{(k+d-1)^2}{4 |\lambda|^2} \left| f^{-\lambda} \times_{\lambda} \varphi_{k-1}(\sqrt{|\lambda|}\cdot)(z) \right|^2 |\lambda|^{2d} \, d\lambda \, dz \\
&+ \int_{\C^d} \int_{\R\setminus\{0\}} \frac{(k+1)^2}{4 |\lambda|^2} \left| f^{-\lambda} \times_{\lambda} \varphi_{k+1}(\sqrt{|\lambda|}\cdot)(z) \right|^2 |\lambda|^{2d} \, d\lambda \, dz. 
\end{align*}

Hence, we have that 
\begin{align}\label{Leibniz-Hn-ineq2}
& \sum_{k \geq 0} \int_{\H^d} |(f \ast (\omega e_k))(z,t)|^2 (2k+d)^{-2} \, dz \, dt \\
\nonumber &\quad \quad \lesssim \sum_{k \geq 0} \int_{\C^d} \int_{\R\setminus\{0\}} \frac{1}{|\lambda|^2} \left| f^{-\lambda} \times_{\lambda} \varphi_{k}(\sqrt{|\lambda|}\cdot)(z) \right|^2 \, d\lambda \, dz \\ 
\nonumber &\quad \quad = c_d \int_{\C^d} \int_{\R\setminus\{0\}} \frac{1}{|\lambda|^2} \left| f^{\lambda} (z) \right|^2 \, d\lambda \, dz. 
\end{align}

Now, for estimating the sum (over $k \geq 0$) of the third term on the right hand side of the inequality \eqref{Leibniz-Hn-ineq}, it is sufficient to analyse  
\begin{equation}\label{one} 
\int_{\H^d} | (\zeta_j f) \ast (\bar{\zeta_j} e_k)(z,t)|^2 \,dz \,dt 
\end{equation} 
for a fixed $1 \leq j \leq d$, as the estimates for the terms with other $j$'s follow analogously. 

We will now simplify the expression of $\bar{\zeta_j} e_k(z,t).$ Using \eqref{laguerre-expansion}, we have 
$$\bar{\zeta_j} e_k(z,t) = \int_{\R} \bar{z}_j \varphi_k(\sqrt{|\lambda|}z)|\lambda|^d e^{i \lambda t} \, d\lambda = (2 \pi)^{d/2} \sum_{|\alpha|=k} \int_{\R} \bar{z_j} \phi_{\alpha,\alpha}(\sqrt{|\lambda|}z) |\lambda|^d e^{i \lambda t} \, d\lambda.$$

We first look at $\bar{z}_j \phi_{\alpha,\alpha} (\sqrt{|\lambda|}z).$ It is known (see, for example, equation (1.3.24) on page $18$ in \cite{ST1}) that 
$$\bar{z}_j \phi_{\alpha,\alpha} (\sqrt{|\lambda|}z) = -i|\lambda|^{-1/2} \left\{(2\alpha_j+2)^{1/2} \phi_{\alpha+e_j,\alpha} (\sqrt{|\lambda|}z) - (2\alpha_j)^{1/2} \phi_{\alpha,\alpha-e_j} (\sqrt{|\lambda|}z) \right\}. $$
From this, we get that $\left(\bar{\zeta}_j e_k\right)^{\lambda}(z)$ is a constant multiple of 
$$|\lambda|^{-1/2} \sum_{|\alpha|=k} \left\{ (2\alpha_j+2)^{1/2} \phi_{\alpha+e_j,\alpha}(\sqrt{|\lambda|}z) - (2\alpha_j)^{1/2} \phi_{\alpha,\alpha-e_j} (\sqrt{|\lambda|}z) \right\} |\lambda|^d.$$

Applying the Plancherel theorem in the $t$-variable in \eqref{one} and making use of the above identity for $\left(\bar{\zeta}_j e_k\right)^{\lambda}(z)$, we get 
\begin{align*}
& \int_{\H^d}| (\zeta_j f) \ast (\bar{\zeta_j} e_k)(z,t)|^2 \, dz \,dt \\ 
& \quad \quad = c_d \int_{\C^d}\int_{\R}\left|\left((\zeta_j f) \ast (\bar{\zeta_j} e_k)\right)^{\lambda}(z)\right|^2 \, dz \, d\lambda\\
& \quad \quad = c_d \int_{\C^d}\int_{\R\setminus\{0\}}|(\zeta_j f)^{-\lambda} \times_{\lambda} (\bar{\zeta_j} e_k)^{\lambda}(z)|^2 \, dz \, d\lambda \\
& \quad \quad \lesssim \int_{\C^d}\int_{\R\setminus\{0\}} \left| \sum_{|\alpha|=k}(2\alpha_j+2)^{1/2} (\zeta_j f)^{-\lambda} \times_{\lambda} \phi_{\alpha+e_j,\alpha}(\sqrt{|\lambda|}\cdot)(z) \right|^2 |\lambda|^{2d-1}\,dz \, d\lambda \\ 
& \quad \quad \quad + \int_{\C^d}\int_{\R\setminus\{0\}} \left| \sum_{|\alpha|=k}(2\alpha_j)^{1/2} (\zeta_j f)^{-\lambda} \times_{\lambda}\phi_{\alpha,\alpha-e_j}(\sqrt{|\lambda|}\cdot)(z) \right|^2 |\lambda|^{2d-1}\,dz \,d\lambda \\
& \quad \quad \lesssim (2k+d) \sum_{|\alpha|=k} \int_{\C^d}\int_{\R\setminus\{0\}} \left| (\zeta_j f)^{-\lambda} \times_{\lambda}\phi_{\alpha+e_j,\alpha}(\sqrt{|\lambda|}\cdot)(z) \right|^2 |\lambda|^{2d-1}\,dz \,d\lambda \\ 
& \quad \quad \quad + (2k+d) \sum_{|\alpha|=k} \int_{\C^d}\int_{\R\setminus\{0\}} \left| (\zeta_j f)^{-\lambda} \times_{\lambda}\phi_{\alpha,\alpha-e_j}(\sqrt{|\lambda|}\cdot)(z) \right|^2 |\lambda|^{2d-1}\,dz \,d\lambda,
\end{align*}
where the last inequality follows from the identity \eqref{twisted} together with the fact that $|\lambda|^{d/2} \phi_{\alpha,\beta}(\sqrt{|\lambda|}\cdot)$ are orthonormal in $L^2(\C^d)$. 

From the above inequality, we get
\begin{align}\label{Leibniz-Hn-ineq3}
& \sum_{k=0}^\infty \int_{\H^d}| (\zeta_j f) \ast (\bar{\zeta_j} e_k)(z,t)|^2 (2k+d)^{-2} \, dz \,dt \\ 
\nonumber & \quad \quad \lesssim \sum_{k=0}^\infty \sum_{|\alpha|=k} \int_{\C^d}\int_{\R\setminus\{0\}} \left| (\zeta_j f)^{-\lambda} \times_{\lambda}\phi_{\alpha+e_j,\alpha}(\sqrt{|\lambda|}\cdot)(z) \right|^2 |\lambda|^{2d-1}\,dz \,d\lambda \\ 
\nonumber & \quad \quad \quad + \sum_{k=0}^\infty  \sum_{|\alpha|=k} \int_{\C^d}\int_{\R\setminus\{0\}} \left| (\zeta_j f)^{-\lambda} \times_{\lambda}\phi_{\alpha,\alpha-e_j}(\sqrt{|\lambda|}\cdot)(z) \right|^2 |\lambda|^{2d-1}\,dz \,d\lambda \\
\nonumber & \quad \quad \sim \int_{\C^d} \int_{\R\setminus\{0\}} \frac{1}{|\lambda|}  |z|^2 \left| f^{\lambda} (z) \right|^2 \, d\lambda \, dz.
\end{align}

Similar calculations could be performed to estimate the sum (over $k \geq 0$) of the fourth term on the right hand side of the inequality \eqref{Leibniz-Hn-ineq}. Finally, substituting estimates from \eqref{Leibniz-Hn-ineq1}, \eqref{Leibniz-Hn-ineq2} and \eqref{Leibniz-Hn-ineq3} into \eqref{Leibniz-Hn-ineq}, we have 
\begin{align}\label{interpolation-eq2}
& \quad \quad \lefteqn{\sum_{k \geq 0} \int_{\H^d} |(f \ast e_k)(z,t)|^2 (2k+d)^{-2} |t|^{2} \, dz \, dt} \\
\nonumber
&\quad \lesssim \int_{\H^d} |f(z,t)|^2 |t|^2 \, dz \, dt + \int_{\C^d} \int_{\R\setminus\{0\}} \frac{\left| f^{\lambda} (z) \right|^2}{|\lambda|^2} \, d\lambda \, dz + \int_{\C^d} \int_{\R\setminus\{0\}} \frac{|z|^2 \left| f^{\lambda} (z) \right|^2}{|\lambda|} \, d\lambda \, dz.
\end{align}

Finally, we claim that for $0< \beta < 2$, the estimate \eqref{interpolated-eq} could be proved by invoking interpolation between the parameters $\beta = 0$ and $\beta = 2$ in the inequalities \eqref{interpolation-eq1} and \eqref{interpolation-eq2}. We postpone the proof of this interpolation to the next section (appendix). This completes the proof of Lemma \ref{A4}. 
\end{proof}

\section{Appendix}\label{appendix}
We made use of an interpolation of certain Hilbert spaces while claiming inequality \eqref{interpolated-eq} in the proof of Lemma \ref{A4}. We shall prove the same here. Consider following Hilbert spaces: 
\begin{align}\label{frac-Hilbert} 
A_s = \left\{f \in D^\prime(\mathbb{R}^{2d+1}) : \|f\|_s^2 < \infty \right\},
\end{align}
for each $0 \leq s \leq 1$, where $D^\prime(\mathbb{R}^{2d+1})$ is the space of tempered distributions on $\R^{2d+1}$ and 
\begin{align*}
\|f\|_s^2 &= \int_{\mathbb{R}^{2d} \times \mathbb{R} \setminus \{0\}} \frac{|f(z,t)|^2}{|t|^{2s}} \, dz \, dt + \int_{\mathbb{R}^{2d+1}} \left|D^s f(z,t)\right|^2  \, dz \, dt \\
&\quad + \int_{\mathbb{R}^{2d} \times \mathbb{R} \setminus \{0\}} \frac{|z|^{2s}}{|t|^s} \left|f(z,t)\right|^2  \, dz \, dt.
\end{align*}

Here $D^s$, for $0<s<1$, denotes the fractional derivative in $t$-variable (see \eqref{def-Fourier-mult}), $D^0$ is the identity operator, whereas $D = D^1$ stands for the distributional derivative in $t$-variable. Note also that $A_0$ coincides with  $L^2 (\mathbb{R}^{2d+1})$. 

For each $0 \leq s \leq 1$, and $k = 0, 1, 2, 3, \ldots$, let $M_k^s = L^2(\C^d \times \R, \omega_k^s(z,t) \, dz \, dt)$, with $\omega_k^s (z, t ) = (2k+d)^{-2s} |t|^{2s}$, and consider following vector valued Hilbert Spaces: 
$$l_2(M_k^s) = \left\{a = (a_j^s)_{j \geq 0} ~:~ a_j^s \in M_k^s, \|a\|^2_{l_2(M_k^s)} := \sum_{j=0}^\infty \|a_j^s\|_{M_k^s}^2 < \infty \right\}.$$

Define the linear operator $T$ on $A_0 + A_1$ by
$$Tf = (\tilde{f} \ast e_k)_{k=0}^\infty,$$
where $\tilde{f}(z,t) = \int_{\mathbb{R}} f(z,\lambda) e^{it \lambda} \, d\lambda$.

Then, as shown in inequalities \eqref{interpolation-eq1} and \eqref{interpolation-eq2} in the proof of Lemma \ref{A4}, $T$ maps $A_0$ and $A_1$ boundedly into $l_2(M_k^0)$ and $l_2(M_k^1)$ respectively, that is, 
\Bea 
\|Tf\|_{l_2(M_k^0)} &\lesssim& \|f\|_0, \\
\|Tf\|_{l_2(M_k^1)} &\lesssim& \|f\|_1.
\Eea 

For each $0 < s < 1$, the complex interpolation of the spaces $l_2(M_k^0)$ and $l_2(M_k^1)$ is (see Page 121, Section 1.18.1 of \cite{HT}):
$$\left[ l_2(M_k^0), l_2(M_k^1) \right]_{s} =  l_2\left(\left[M_k^0, M_k^1\right]_{s}\right).$$

Now, the interpolation of the spaces $M_k^0$ and $M_k^1$ is well known and follows from general theory of interpolation of weighted $L^p$ spaces (see, for example, 5.4, Page 241, Chapter 5 of \cite{SW}), and for each $0 < s < 1$ we have 
$$\left[M_k^0, M_k^1\right]_{s} = M_k^s.$$ 

Finally, if we could show that $A_s$, for $0<s<1$, is the complex interpolation space of $A_0$ and $A_1$, then we will have that $T$ is bounded from $A_s$ into $l_2(M_k^s)$, and this will complete the proof of the inequality \eqref{interpolated-eq}. We now explain the complex interpolation of $A_0$ and $A_1$. 

\begin{thm}\label{thm:interpolation} 
Let the Hilbert spaces $A_s$, for $0 \leq s \leq 1,$ be as defined in \eqref{frac-Hilbert}. For each $0<s<1$, the complex interpolation of the pair $(A_0, A_1)$ is $A_s$.  
\end{thm}

\begin{proof} 
We closely follow the proof of \cite{Tr}, and rewrite it to suit to our context. For each $j \in \mathbb{Z}$, consider intervals $\Omega_j^+ = \left(2^{-j-3}, 2^{-j+3}\right),$ and $\Omega_j^- = \left(-2^{-j+3}, -2^{-j-3}\right).$ Clearly, these intervals have finite overlap, $(0, \infty) = \cup_{j \in \mathbb{Z}} \Omega_j^+$, and $(-\infty, 0) = \cup_{j \in \mathbb{Z}} \Omega_j^-$. Choose and fix $\phi \in C_c^\infty \left(\Omega_0^+ \right)$ with the property that $\phi \equiv 1$ on the interval $[\frac{1}{2}, 2]$, $0 < \phi \leq 1$ on the interval $(\frac{1}{4}, 4),$ and $0$ elsewhere. Now, define $\phi_j (t) = \phi\left(2^{j}t\right)$ for $j \in \mathbb{Z}$. Finally, for each $j \in \mathbb{Z}$, define  
\Bea
\widetilde{\phi}_j(t) = 
\begin{cases}
\frac{\phi_j(-t)}{\sum_{k \in \mathbb{Z}} \phi_k(-t)} &; \mbox{ for } t \in (-\infty, 0)\\ 
\frac{\phi_j(t)}{\sum_{k \in \mathbb{Z}} \phi_k(t)} &; \mbox{ for } t \in (0,\infty) \, . 
\end{cases}
\Eea

Clearly, $\widetilde{\phi}_j \in C_c^\infty \left(\Omega_j^+ \cup \Omega_j^- \right)$, and $\sum_j \widetilde{\phi}_j \equiv 1$. Moreover, for any $t \in (0, \infty)$, 
\Bea 
\widetilde{\phi}_j^\prime(t) &=& \frac{\phi_j^\prime(t)\sum_k \phi_k(t) - \phi_j(t)\sum_k \phi_k^\prime(t)}{\left(\sum_k \phi_k(t)\right)^2} \\
&=& \frac{\phi_j^\prime(t)\sum_{k= j-3}^{j+3} \phi_k(t) - \phi_j(t)\sum_{k= j-3}^{j+3} \phi_k^\prime(t)}{\left(\sum_{k \in \mathbb{Z}} \phi_k(t)\right)^2} \lesssim 2^j.
\Eea

In the last inequality, we used the fact that $\sum_{k \in \mathbb{Z}} \phi_k$ is uniformly bounded from below on $(0,\infty)$. In fact, $\sum_{k \in \mathbb{Z}} \phi_k(t) \geq 1$. Similar estimate holds true on $(-\infty, 0)$. 

For each $0 \leq s \leq 1$, let us now consider the spaces 
$$W_2^s = \left\{f \in D^\prime(\mathbb{R}^{2d+1}) : \|f\|_{s,*}^2 < \infty \right\},$$
where $\|f\|_{s,*}^2$ is defined to be 
$$\sum_{j \in \Z} \int_{\mathbb{R}^{2d+1}} \left(2^{2js} |(\widetilde{\phi}_j  f)(z,t)|^2 + |D^s (\widetilde{\phi}_j f)(z,t)|^2 + 2^{js} |z|^{2s} |(\widetilde{\phi}_j f)(z,t)|^2 \right) \, dz \, dt.$$

We shall show that 
\begin{enumerate}
\item $W_2^s$ is a Banach space. 
\item $C_c^\infty(\mathbb{R}^{2d} \times \mathbb{R} \setminus \{0\})$ is dense in $W_2^s$ in the $\|\cdot\|_{s,*}$ norm. 
\item $W_2^s$ is the completion of $C_c^\infty(\mathbb{R}^{2d} \times \mathbb{R} \setminus \{0\})$ in the $\|\cdot\|_s$ norm. In other words, the spaces $A_s$ and $W_2^s$ are identical with norm equivalence. 
\end{enumerate} 

It is easy to note that $W_2^0$ is noting but $A_0$ with norm equivalence. So, we only need to study the case when $0 < s \leq 1$.

\medskip \noindent {\bf (1) $W_2^s$ is a Banach space.}

It could be proved using standard arguments, so we omit the proof. 

\medskip \noindent {\bf (2) $C_c^\infty(\mathbb{R}^{2d} \times \mathbb{R} \setminus \{0\})$ is dense in $W_2^s$ in the $\|\cdot\|_{s,*}$ norm.}

Take $f \in W_2^s$, then for any $M \in \mathbb{N}$ large enough, $\| f - f \sum\limits_{|k| \leq M+3} \widetilde{\phi}_k \|_{s,*}^2 $ equals to 
\begin{align*}
&\sum_j 2^{2js} \|\sum_{|k| > M+3} \widetilde{\phi}_k \widetilde{\phi}_j f\|_{L^2(\mathbb{R}^{2d+1})}^2 + \sum_j \|D^s (\sum_{|k| > M+3} \widetilde{\phi}_k \widetilde{\phi}_j f)\|_{L^2(\mathbb{R}^{2d+1})}^2\\
&\quad \quad + \sum_j 2^{js} \||z|^{s} \sum_{|k| > M+3} \widetilde{\phi}_k \widetilde{\phi}_j f\|_{L^2(\mathbb{R}^{2d+1})}^2  \\
=& \sum_{|j| > M} 2^{2js} \|\sum_{|k| > M+3} \widetilde{\phi}_k \widetilde{\phi}_j f\|_{L^2(\mathbb{R}^{2d+1})}^2 + \sum_{|j| > M} \|D^s (\sum_{|k| > M+3} \widetilde{\phi}_k \widetilde{\phi}_j f)\|_{L^2(\mathbb{R}^{2d+1})}^2 \\
&\quad \quad + \sum_{|j| > M} 2^{js} \| |z|^{s} \sum_{|k| > M+3} \widetilde{\phi}_k \widetilde{\phi}_j f\|_{L^2(\mathbb{R}^{2d+1})}^2,
\end{align*}
using the fact that $\widetilde{\phi}_j \widetilde{\phi}_k = 0$ for any $j$ and $k$ with $|j-k| > 3.$ The above estimate is bounded by
\begin{align*} 
& \sum_{|j| > M} \sum_{|k-j| \leq 3} 2^{2js} \|\widetilde{\phi}_k \widetilde{\phi}_j f\|_{L^2(\mathbb{R}^{2d+1})}^2 + \sum_{|j| > M} \sum_{|k-j| \leq 3} \|D^s (\widetilde{\phi}_k \widetilde{\phi}_j f)\|_{L^2(\mathbb{R}^{2d+1})}^2 \\ 
& \quad \quad + \sum_{|j| > M} \sum_{|k-j| \leq 3}  2^{js} \| |z|^s \widetilde{\phi}_k \widetilde{\phi}_j f\|_{L^2(\mathbb{R}^{2d+1})}^2.
\end{align*}

Clearly, the first and third terms of the above sum tend to $0$ as $M \to \infty$. For the second term, we claim that 
$$\|D^s (\widetilde{\phi}_k f)\|_{L^2(\mathbb{R}^{2d+1})}^2 \lesssim 2^{2ks} \| f\|_{L^2(\mathbb{R}^{2d+1})}^2 + \|D^s f\|^2_{L^2(\mathbb{R}^{2d+1})}.$$

From these estimates with $\widetilde{\phi}_j f$ in place of $f$, it follows immediately that 
$$\lim_{M \to \infty} \| f - f \sum_{|k| \leq M+3} \phi_k\|_{s,*} = 0.$$

The claimed estimate for $D^s (\widetilde{\phi}_k f)$, for $s=1,$ follows once we apply Leibniz rule for differentiation and then using estimates of derivatives of $\widetilde{\phi}_k$'s. For $0 < s < 1$, we assume the claimed estimate for now. In fact, in the next step, we estimate the full sum $\sum_k \|D^s (\widetilde{\phi}_k f)\|_{L^2(\mathbb{R}^{2d+1})}^2$. In those detailed calculations, one could easily verify that the claimed estimate for each fixed $k$ also holds. Thus, we have so far shown that every function $f$ in $W^2_s$ could be approximated by a sequence of functions supported in sets of the form $\mathbb{R}^{2d} \times K$ with $K$ compact in $\mathbb{R} \setminus \{0\}$. Finally, one could use standard methods of approximation to complete the claim that $C_c^\infty(\mathbb{R}^{2d} \times \mathbb{R} \setminus \{0\})$ is dense in $W_2^s$ in the $\|\cdot\|_{s,*}$ norm. 

\medskip \noindent {\bf (3) $W_2^s$ is the completion of $C_c^\infty(\mathbb{R}^{2d} \times \mathbb{R} \setminus \{0\})$ in the $\|\cdot\|_s$ norm.}

It is straightforward to verify that 
\begin{align*}
\int_{\mathbb{R}^{2d} \times \mathbb{R} \setminus \{0\}} |f(z,t)|^2 \frac{1}{|t|^{2s}} \, dz \, dt & \sim \sum_j 2^{2js} \int_{\mathbb{R}^{2d+1}} | (\widetilde{\phi}_j f)(z,t)|^2\, dz \, dt, \\
\textup{ and } \int_{\mathbb{R}^{2d} \times \mathbb{R} \setminus \{0\}} |f(z,t)|^2 \frac{|z|^{2s}}{|t|^{s}} \, dz \, dt & \sim \sum_j 2^{js} \int_{\mathbb{R}^{2d+1}} |z|^{2s} | (\widetilde{\phi}_j f)(z,t)|^2\, dz \, dt. 
\end{align*}

Next, writing $\Omega_j = \Omega_j^+ \cup \Omega_j^-$, we perform following calculations to estimate $\sum_j \|D^s (\widetilde{\phi}_j f)\|_{L^2(\mathbb{R}^{2d+1})}$. For $0<s<1$, 
\begin{align*} 
\sum_j \|D^s (\widetilde{\phi}_j f)\|_{L^2(\mathbb{R}^{2d+1})}^2 & =  C \sum_j \int_{\mathbb{R}^{2d}} \int_{\mathbb{R}^2} \frac{|(\widetilde{\phi}_j f)(z,t_1) - (\widetilde{\phi}_j f) (z,t_2)|^2}{|t_1 - t_2|^{1+2s}} \, dz \, dt_1 \, dt_2 \\ 
& \lesssim \sum_j \int_{\mathbb{R}^{2d}} \int_{\mathbb{R} \times \Omega_j} \frac{|(\widetilde{\phi}_j f)(z,t_1) - (\widetilde{\phi}_j f) (z,t_2)|^2}{|t_1 - t_2|^{1+2s}} \, dz \, dt_1 \, dt_2,
\end{align*}
which is dominated by
\begin{align*} 
&\sum_j \int_{\mathbb{R}^{2d}} \int_{\Omega_j \times \Omega_j} |\widetilde{\phi}_j (t_1)|^2 \frac{|f(z,t_1) - f (z,t_2)|^2}{|t_1 - t_2|^{1+2s}} \, dz \, dt_1 \, dt_2 \\
& \quad \quad + \sum_j \int_{\mathbb{R}^{2d}} \int_{\Omega_j \times \Omega_j}  \frac{|\widetilde{\phi}_j(t_1) - \widetilde{\phi}_j (t_2)|^2}{|t_1 - t_2|^{1+2s}} |f(z,t_2)|^2 \, dz \, dt_1 \, dt_2 \\
& \quad \quad + \sum_j \int_{\mathbb{R}^{2d}} \int_{\Omega_j^c \times \Omega_j}  \frac{|\widetilde{\phi}_j (t_2)|^2}{|t_1 - t_2|^{1+2s}} |f(z,t_2)|^2 \, dz \, dt_1 \, dt_2 \\
=: &~~ I + II + III.
\end{align*}

We now estimate sums $I, II$, and $III$ as follows. 
\begin{align*}
I &\leq \sum_j \int_{\mathbb{R}^{2d}} \int_{\Omega_j \times \Omega_j} \frac{|f(z,t_1) - f (z,t_2)|^2}{|t_1 - t_2|^{1+2s}} \, dz \, dt_1 \, dt_2 \\
&\lesssim \int_{\mathbb{R}^{2d}} \int_{\mathbb{R}^2} \frac{|f(z,t_1) - f (z,t_2)|^2}{|t_1 - t_2|^{1+2s}} \, dz \, dt_1 \, dt_2 \\
&= C \|D^s f\|_{L^2(\mathbb{R}^{2d+1})}^2. 
\end{align*}

\begin{align*}
II &= \sum_j \int_{\mathbb{R}^{2d} \times \Omega_j} \left(\int_{\Omega_j}  \frac{|\widetilde{\phi}_j(t_1) - \widetilde{\phi}_j (t_2)|^2}{|t_1 - t_2|^{1+2s}}  \, dt_1\right) |f(z,t_2)|^2 \, dz \, dt_2 \\
&\lesssim \sum_j 2^{2js} \int_{\mathbb{R}^{2d} \times \Omega_j} |f(z,t_2)|^2 \, dz \, dt_2 \\
&\lesssim \int_{\mathbb{R}^{2d} \times \mathbb{R} \setminus \{0\}} |f(z,t)|^2 \frac{1}{|t|^{2s}} \, dz \, dt. 
\end{align*}

Here we have used the estimate 
$$\int_{\Omega_j}  \frac{|\widetilde{\phi}_j(t_1) - \widetilde{\phi}_j (t_2)|^2}{|t_1 - t_2|^{1+2s}}  \, dt_1 \lesssim 2^{2js}$$
which holds uniformly for $ t_2 \in \Omega_j$, and the proof of this estimate is 
\begin{align*}
\int_{\Omega_j}  \frac{|\widetilde{\phi}_j(t_1) - \widetilde{\phi}_j (t_2)|^2}{|t_1 - t_2|^{1+2s}}  \, dt_1 &= \int_{\Omega_j} \left| \frac{\widetilde{\phi}_j(t_1) - \widetilde{\phi}_j (t_2)}{t_1 - t_2|}\right|^2 |t_1 - t_2|^{1-2s} \, dt_1 \\
&\lesssim \|\widetilde{\phi}_j^\prime\|^2_\infty \int_{\Omega_j} |t_1 - t_2|^{1-2s} \, dt_1 \\
&\sim 2^{2js}. 
\end{align*}

Recalling that $\Omega_j = (-2^{-j+3}, -2^{-j-3}) \cup (2^{-j-3}, 2^{-j+3}),$ and $\widetilde{\phi}_j$ are supported on $[-2^{-j+2}, -2^{-j-2}] \cup [2^{-j-2}, 2^{-j+2}]$, one can perform calculations similar to the above ones to verify that 
$$\int_{\Omega_j}  \frac{1}{|t_1 - t_2|^{1+2s}}  \, dt_1 \lesssim 2^{2js},$$ 
and thus, 
\begin{align*}
III &\lesssim \sum_j 2^{2js} \int_{\mathbb{R}^{2d} \times \Omega_j} |f(z,t_2)|^2 \, dz \, dt_2 \\
&\lesssim \int_{\mathbb{R}^{2d} \times \mathbb{R} \setminus \{0\}} |f(z,t)|^2 \frac{1}{|t|^{2s}} \, dz \, dt. 
\end{align*}

So, we have have shown that for any $0<s<1$,
$\sum_j \|D^s (\widetilde{\phi}_j f)\|_{L^2(\mathbb{R}^{2d+1})}$ is bounded by a multiple of $\|f\|_{s}$. 

Summarizing the above, we have that for any $0<s<1$, 
$$\|f\|_{s,*} \lesssim \|f\|_{s}.$$

For $s=1$, while analysing $\|f\|_{1,*}$, one could simply apply the Leibniz formula for differentiation in the first two terms and then make use of the estimates of $\widetilde{\phi}_j$'s together with the fact that these are supported in $\Omega_j$'s, and the above analysis for the fractional differentiation for the third term to easily verify that 
$$\|f\|_{1,*} \lesssim \|f\|_{1}.$$

On the other hand, for $0<s<1$, 
\begin{align*} 
& \lefteqn{\int_{\mathbb{R}^{2d+1}} \left|D^s f(z,t)\right|^2  \, dz \, dt } \\
\begin{split}
& \quad = \int_{\mathbb{R}^{2d}} \int_{\mathbb{R}^2} \frac{|f(z,t_1) - f(z,t_2)|^2}{|t_1 - t_2|^{1+2s}} \, dz \, dt_1 \, dt_2 \\ 
& \quad \lesssim \sum_{m} \sum_{l \geq m} \int_{\mathbb{R}^{2d}} \int_{\Omega_m \times \Omega_l} \frac{|\sum_j (f \widetilde{\phi}_j)(z,t_1) - \sum_k (f \widetilde{\phi}_k)(z,t_2)|^2}{|t_1 - t_2|^{1+2s}} \, dz \, dt_1 \, dt_2 \\
& \quad = \sum_{m} \sum_{l \geq m} \int_{\mathbb{R}^{2d}} \int_{\Omega_m \times \Omega_l} \frac{|\sum_{j=m-3}^{m+3} (f \widetilde{\phi}_j)(z,t_1) - \sum_{k=l-3}^{l+3} (f \widetilde{\phi}_k)(z,t_2)|^2}{|t_1 - t_2|^{1+2s}} \, dz \, dt_1 \, dt_2.
\end{split} 
\end{align*}

For each $m$, one could arrange the summand in the above expression in the following manner. For $m \leq l \leq m+9$, write each pair of terms with same index together and the remaining terms separately. For $l > m+9$, there is no common index, and we write each term separately. Finally, one could apply Cauchy Schwarz inequality to verify that the above summation is dominated by 
\begin{align*}
& \sum_{m} \int_{\mathbb{R}^{2d}} \int_{\Omega_m \times \Omega_m} \frac{|(f \widetilde{\phi}_m)(z,t_1) - (f \widetilde{\phi}_m)(z,t_2)|^2}{|t_1 - t_2|^{1+2s}} \, dz \, dt_1 \, dt_2 \\ 
& \quad + \sum_{m} \int_{\mathbb{R}^{2d}} \int_{\Omega_m \times \Omega_m^c} \frac{|(f \widetilde{\phi}_m)(z,t_1)|^2}{|t_1 - t_2|^{1+2s}} \, dz \, dt_1 \, dt_2 
\end{align*}
which is further bounded from above by 
$$\sum_{m} \|D^s(f \widetilde{\phi}_m)\|^2_{L^2(\mathbb{R}^{2d+1})} + \sum_{m} 2^{2ms} \|f \widetilde{\phi}_m\|^2_{L^2(\mathbb{R}^{2d+1})}.$$

This completes the proof of the fact that for $0<s<1$,
$$\|f\|_{s} \lesssim \|f\|_{s,*}.$$

Finally, when $s=1$, one could directly estimate $Df$ in terms of $D(f \widetilde{\phi}_m)$ as follows : 
\begin{align*} 
\left\|\frac{\partial f}{\partial t}\right\|^2_{L^2(\mathbb{R}^{2d+1})} 
&= \int_{\mathbb{R}^{2d+1}} \left|\sum_{m} \frac{\partial (f \widetilde{\phi}_m)}{\partial t} (z,t)\right|^2  \, dz \, dt \\
&\lesssim \sum_{m} \int_{\Omega_m} \left|\frac{\partial (f \widetilde{\phi}_m)}{\partial t} (z,t)\right|^2  \, dz \, dt \\
&\leq \sum_{m} \left\|\frac{\partial (f \widetilde{\phi}_m)}{\partial t}\right\|^2_{L^2(\mathbb{R}^{2d+1})} .
\end{align*}
Hence, for all $0 \leq s \leq 1$, the two spaces $(A_s, \|\cdot\|_s)$ and $(W_2^s, \|\cdot\|_{s,*})$ are identical with norm equivalence. For any $0<s<1$, since $W_2^s$ is the complex interpolation of $W_2^0$ and $W_2^1$ (see Page 121, Section 1.18.1 of \cite{HT}), it follows that $A_s$ is the complex interpolation of $A_0$ and $A_1.$ This completes the proof of Theorem \ref{thm:interpolation}.
\end{proof}

\section*{Acknowledgements}
The authors would like to thank the referee for meticulous reading of the manuscript and for valuable suggestions which have greatly helped in improving the presentation of the paper. The second author is very grateful for the kind hospitality provided at IISER Bhopal where this project was initiated. This work was supported in part by the INSPIRE Faculty Award of the first author from the Department of Science and Technology (DST), Government of India. 

\bibliographystyle{amsplain}

\end{document}